\documentclass[nonatbib,12pt,preprint]{elsarticle}
\pdfoutput=1 
\makeatletter
\let\c@author\relax
\def\ps@pprintTitle{%
 \let\@oddhead\@empty
 \let\@evenhead\@empty
 \def\@oddfoot{}%
 \let\@evenfoot\@oddfoot}
\makeatother

\usepackage[T1]{fontenc}

\usepackage{bm}
\usepackage{amsmath}
\usepackage{mathrsfs}
\usepackage{amsthm}
\usepackage{amssymb}
\usepackage{centernot}
\usepackage{enumerate}
\usepackage{dsfont}
\usepackage{comment}

\usepackage{nicefrac}
\usepackage{tikz}
\definecolor{ao}{rgb}{0.0, 0.5, 0.0}

\usepackage[
    backend=biber,
    citestyle=numeric-comp,
    bibstyle=ieee,
    sorting=nyt,
    isbn=false,
    url=false,
    citecounter=true,
    ]{biblatex}
\AtEveryBibitem{\clearlist{language}}

\usepackage{enumitem}

\usepackage[hidelinks]{hyperref}
\usepackage[capitalize]{cleveref}

\setlist[enumerate,1]{label={(\arabic*)},
                   ref  ={(\arabic*)}}

\newlist{propenum}{enumerate}{1} 
\setlist[propenum]{label=\arabic*), ref=\theproposition~(\arabic*)}
\crefalias{propenumi}{proposition} 

\crefformat{equation}{#2(#1)#3}

\newcommand{\nn}[1]{\mathbb{N}^{#1}}
\newcommand{\rr}[1]{\mathbb{R}^{#1}}

\newcommand{\cc}[1]{\mathbb{C}^{#1}}

\newcommand{\zzp}[1]{\mathbb{Z}_+^{#1}}

\newcommand{\innerProd}[2]{\left\langle{#1},{#2}\right\rangle}
\newcommand{\eabs}[1]{\langle{#1}\rangle}     

\newcommand{\dee}{\mathop{\mathrm{d}\!}}
\newcommand{\Frechet}{Fr\'{e}chet}

\NewDocumentCommand{\abs}{s o m}{
    \ensuremath{
    \IfBooleanTF{#1}
        {\left|}
        {\IfValueTF{#2}
            {#2|}
            {|}
        }
    {#3}
    \IfBooleanTF{#1}
        {\right|}
        {\IfValueTF{#2}
            {#2|}
            {|}
        }
    }
}
\NewDocumentCommand{\seminorm}{s o m o}{
    \ensuremath{
    p
    \IfValueT{#4}{_{#4}}
    \IfBooleanTF{#1}
        {\left(}
        {\IfValueTF{#2}
            {#2)}
            {(}
        }
    {#3}
    \IfBooleanTF{#1}
        {\right)}
        {\IfValueTF{#2}
            {#2)}
            {)}
        }
    }
}
\NewDocumentCommand{\norm}{s o m o}{
    \ensuremath{
    \IfBooleanTF{#1}
        {\left\|}
        {\IfValueTF{#2}
            {#2\|}
            {\|}
        }
    {#3}
    \IfBooleanTF{#1}
        {\right\|}
        {\IfValueTF{#2}
            {#2\|}
            {\|}
        }
    }
    \IfValueT{#4}{_{#4}}
}

\NewDocumentCommand{\Fourier}{s o m}{
    \ensuremath{
        \IfValueTF{#2}{
            \F^{#2}
        }{
            \F
        }
        \IfBooleanTF{#1}{
            \left\{
        }{
            \{
        }
        #3
        \IfBooleanTF{#1}{
            \right\}
        }{
            \}
        }
    }
}
\NewDocumentCommand{\FourierInv}{s m}{
    \IfBooleanTF{#1}{
        \Fourier*[-1]{#2}
    }{
        \Fourier[-1]{#2}
    }
}

\newcommand{\conv}[0]{*}
\newcommand{\chF}[1]{\chi_{#1}}
\newcommand{\ch}[2]{{\chF{#1}}({#2})}

\newcommand{\set}[1]{\mathcal{#1}}

\NewDocumentCommand{\sobolev}{s m o o}{
    \IfBooleanTF{#1}
        {
            H^{#2}
        }
        {
            W^{#2}
        }
        \IfValueT{#3}
            {(#3\IfValueT{#4}{;#4})}
    }
    
\NewDocumentCommand{\schwartz}{o}{\mathscr{S}
    \IfValueT{#1}{(#1)}
}
\NewDocumentCommand{\GS}{s o m o}{\IfBooleanTF{#1}{\ensuremath{\Sigma}}{S}\IfValueT{#2}{_{#2}}^{#3}\IfValueT{#4}{(#4)}}
\NewDocumentCommand{\hormander}{m m o}{\mathbf{S}_{#1}^{#2}\IfValueT{#3}{(#3)}}
\NewDocumentCommand{\F}{}{\mathscr{F}}
\NewDocumentCommand{\FL}{m m o}{\F L^{#1}(#2\IfValueT{#3}{;#3})}
\NewDocumentCommand{\barron}{s m o}{
    \IfBooleanTF{#1}{\set{B}^*}{\set{B}}_{#2}\IfValueT{#3}{(#3)}
}

\newcommand{\nNeuron}{N}

\newcommand{\targetSpace}[0]{\set{V}}

\numberwithin{equation}{section}
\newtheorem{theorem}{Theorem}[section]
\newtheorem{definition}[theorem]{Definition}
\newtheorem{corollary}[theorem]{Corollary}
\newtheorem{lemma}[theorem]{Lemma}

\newtheorem{proposition}[theorem]{Proposition}

\theoremstyle{remark}
\newtheorem{remark}[theorem]{Remark}
\newtheorem*{remark*}{Remark}

\begin{document}

\begin{frontmatter}

\title{Approximation Rates in \Frechet{} Metrics: Barron Spaces, Paley-Wiener Spaces, and Fourier Multipliers}
\author[label1]{Ahmed Abdeljawad\corref{cor1}}
\ead{ahmed.abdeljawad@oeaw.ac.at}
\author[label1]{Thomas Dittrich}
\ead{thomas.dittrich@oeaw.ac.at}
    
\affiliation[label1]{organization={Johann Radon Institute for Computational and Applied Mathematics (RICAM), Austrian Academy of Sciences},
            addressline={Altenberger Straße 69}, 
            city={Linz},
            postcode={4040}, 
            country={Austria}}
	

\begin{abstract}
    Operator learning is a recent development in the simulation of partial differential equations by means of neural networks.
    The idea behind this approach is to learn the behavior of an operator, such that the resulting neural network is an approximate mapping in infinite-dimensional spaces that is capable of (approximately) simulating the solution operator governed by the partial differential equation.
    In our work, we study some general approximation capabilities for linear differential operators by approximating the corresponding symbol in the Fourier domain.
    Analogous to the structure of the class of Hörmander-Symbols, we consider the approximation with respect to a topology that is induced by a sequence of semi-norms.
    In that sense, we measure the approximation error in terms of a \Frechet{} metric, and our main result identifies sufficient conditions for achieving a predefined approximation error.
    Secondly, we then focus on a natural extension of our main theorem, in which we manage to reduce the assumptions on the sequence of semi-norms.
    Based on existing approximation results for the exponential spectral Barron space, we then present a concrete example of symbols that can be approximated well.
\end{abstract}



\begin{keyword}
Neural Networks\sep Approximation Rates\sep Operator Learning\sep Symbol Approximation\sep Barron Spaces\sep \Frechet{} Spaces\sep Paley-Wiener Spaces
\MSC[2020] 41A25\sep 41A46\sep 41A65\sep 46E10\sep 68T05\sep 68T07
\end{keyword}

\end{frontmatter}

\section{Introduction}\label{sec:intro}

The field of \emph{Neural Operator} learning has gained a surge of attention over the last four years \cite{Kovachki23NeuralOperatorLearning,DeRyck22Genericboundsapproximation, Schwab23Deepoperatornetwork, Hua23Basisoperatornetwork, Lanthaler23OperatorlearningPCANet, Lanthaler23OperatorlearningPCANet, Kovachki24OperatorLearningAlgorithms,Lu21LearningNonlinearOperators, Li20FourierNeuralOperator, Boulle23MathematicalGuideOperator}.
Unlike traditional numerical approaches, which solve partial differential equations (PDEs) or other functional relationships explicitly, operator learning leverages data-driven methods, particularly neural operators, to approximate mappings between infinite-dimensional spaces. Neural operator architectures, including Deep Operator Networks (DeepONet) \cite{Lu21LearningNonlinearOperators, Lanthaler22ErrorEstimatesDeepONets}, Fourier Neural Operators (FNOs) \cite{Li20FourierNeuralOperator},  Principal Component Analysis Neural Networks  (PCA-Net) \cite{Lanthaler23OperatorlearningPCANet}, Fourier Neural Mappings (FNMs) \cite{Huang24operatorlearningperspective}, extend the capabilities of deep learning to handle operators efficiently. 
With growing theoretical and empirical evidences of success, operator learning has become a very promising surrogate model to handle infinite dimensional problems.

The starting point for operators approximation with neural networks was by Chen and Chen  in 1995  \cite{Chen95UniversalApproximationNonlinear}, where a universal approximation theorem was proved for operators approximation through neural networks where the  operators assumed to be continuous on compact set with a compact range.
Roughly, in \cite{Chen95UniversalApproximationNonlinear}, authors showed that if $V$ is a compact subset of $C(K_1)$ and $K_1$ is a compact set of a given Banach space $X$, and $K_2$ is a compact set in $\rr{d}$ then for any non linear continuous operator $G: V\to C(K_2)$ there exists a neural network \(\Psi[\theta]\), such that 
\[
\begin{gathered}
\abs{G(u)(y)-\Psi[\theta](u)(y)}<\epsilon
\end{gathered}
\]
where the network has the following form
\[
\Psi[\theta](u)(y)
=
\sum_{k=1}^N \sum_{i=1}^M c_i^k g\left(\sum_{j=1}^m \xi_{i j}^k u\left(x_j\right)+\zeta_i^k\right)
\cdot g\left(\omega_k \cdot y+b_k\right),
\]
and $g$ is an activation function.
The neural operator surrogate model \(\Psi[\theta]\) was extended to the so-called DeepONet in \cite{Lu21LearningNonlinearOperators}, where the architecture of the module is defined in terms of an encoder, decoder operators, combined with a finite-dimensional neural network between the latent finite-dimensional spaces.
More explicitly, the encoder can be seen as a general linear map $\mathcal{E}: \mathcal{U}\to \rr{n}, u\mapsto \mathcal{E}(u)$ for example encoding can be done by point evaluation at distinct ``sensor points" $x_k$, that is, $\mathcal{E}(u)= \{u(x_k)\}_{k=1}^n$ .
Whereas the decoder or ``reconstruction" $\mathcal{D}: \rr{m}\to \mathcal{Y}$  is given by expansion with respect to a neural network basis $\tau_1, \dots, \tau_m\in\mathcal{Y}$. Whereas, the finite-dimensional neural network between the latent finite-dimensional spaces can be defined as a parametrized function $\beta: \rr{n}\times \Theta_\beta \to \rr{m}$. Note that $\mathcal{U}$ and $\mathcal{Y}$ are Hilbert spaces. 
The parametrized DeepONet can be written as
\begin{align*}
    \Psi^*(u ; \theta)(y)=\sum_{j=1}^{m} \beta_j\left(\mathcal{E}(u) ; \theta_\beta\right) \tau_j\left(y ; \theta_\tau\right)
\end{align*}
where the branch network $\beta$ processes the input function, and the trunk network $\tau$ processes the coordinates.
A significant difference from shallow neural networks used for function approximation is the substitution of the coefficient term with a branch network, which itself shares similarities with a shallow neural network in structure.
Universal approximation results for DeepONets in different settings was established in \cite{Lu21LearningNonlinearOperators, Lanthaler22ErrorEstimatesDeepONets}.
For different architectures, universal approximation results have been also investigated e.g., \cite{Subedi24ErrorBoundsLearning,Kovachki23NeuralOperatorLearning,Lanthaler23OperatorlearningPCANet, Prasthofer22VariableInputDeepOperator,Jin22MIONetLearningMultipleInput,Lanthaler23NonlocalityNonlinearityImplies,Schwab23Deepoperatornetwork,Calvello24ContinuumAttentionNeural}.
A central argument in demonstrating that neural operators are universal approximators lies in the functional approximation, which plays a crucial role in the proof, see e.g., \cite{Chen95UniversalApproximationNonlinear}. Functional approximation was also investigated in \cite{Mhaskar97NeuralNetworksFunctional} where authors  uniformly approximated a class of nonlinear, continuous functionals defined on $L^p ([-1, 1]^d)$ through neural networks.

Just as the approximation of functionals leads to the approximation of operators, it can be observed that in certain scenarios where operators can be characterized through specific functions, it is sufficient to approximate these functions. Therefore, by definition, this approach inherently results in the approximation of the operator itself. These types of operators are typically associated with a symbol and can be identified through it. They are often, but not exclusively, linear operators.
In practice, it is most interesting to approximate non-linear operators. However, for the purpose of theoretical analysis, we focus on linear operators and delve into the approximation in the general setting of \Frechet{} error metrics.
Namely, we consider operators, which can be represented in the Fourier domain in terms of a symbol and analyze the convergence of neural network approximation classes in the topology of the space of symbols.
The initial starting point for the symbols is the H\"{o}rmander class $\mathcal{S}_{\rho,\delta}^m$ of symbols for pseudodifferential operators.
\begin{definition}[{H\"{o}rmander Symbols \cite{Hormander98AnalysisLinearPartial}}]
For some open $\set{U}\subset{\rr{M}}$, $m\in\rr{}$, $0<\rho\leq1$, and $0\leq \delta<1$, the class $\mathcal{S}_{\rho,\delta}^m(\set{U}\times\rr{N})$ is the set of all $a\in C^\infty(\set{U}\times\rr{N})$ such that
\begin{align}
\label{eq:hoermander_condition}
    \abs{\partial_x^\beta\partial_\xi^\alpha a(x,\xi)}
    \leq C_{\alpha,\beta}(1+\abs{\xi})^{m-\rho\abs{\alpha}+\delta\abs{\beta}}
\end{align}
for some constant $C_{\alpha,\beta}$ and all $(x,\xi)\in\set{U}\times\rr{N}$ and all multiindices $\alpha\in\nn{M}$ and $\beta\in\nn{N}$.
\end{definition}
According to \cite{Hormander07AnalysisLinearPartial}, this class is a \Frechet{} space with semi-norms given by the smallest possible constants that can be used in \cref{eq:hoermander_condition}.
In other words, the semi-norms are
\begin{align*}
    p_{\alpha,\beta}(a)=\sup_{(x,\xi)\in\set{U}\times\rr{N}}\abs{(1+\abs{\xi})^{-m+\rho\abs{\alpha}-\delta\abs{\beta}}\partial_x^\beta\partial_\xi^\alpha a(x,\xi)}.
\end{align*}

By making this transition to approximate symbols in the Fourier domain instead of operators in the original domain, we can reduce the infinite-dimensional problem to a finite-dimensional function approximation problem.
In order to measure the approximation error, we consider a sequence of separating semi-norms $\{p_i\}_{i=1}^\infty$ on a vector space $\set{V}$ and use the \Frechet{} metric 
\begin{align}
\label{eq:frechet_metric}
    d_\set{V}(x-y):=\sum_{i\in\zzp{}}\frac{1}{2^i}\frac{p_i(x-y)}{1+p_i(x-y)},
\end{align}
which, according to \cite[Remark 1.38(c)]{Rudin91FunctionalAnalysis}, induces the same topology as the sequence of semi-norms.
Note that, by employing the \Frechet{} framework, this approach is more general than any Banach space framework. To maximize expressivity, it is beneficial to consider the widest possible class of function spaces, which enhances the versatility and applicability of the approximation results. Hence, the use of the \Frechet{} metric as an error measure enables the analysis and approximation in broader settings. This makes the approach particularly well-suited for universal approximation arguments, especially  when aiming for maximum expressivity and applicability beyond the limitations of Banach spaces.

In a similar fashion \cite{Benth23NeuralNetworksFrechet} addressed approximation of functions from \Frechet{} spaces, equipped with this type of metric, to a Banach space.
For doing so, they extend several properties and results for neural networks to the infinite-dimensional setting.
These extensions include the notion of sigmoidal activation functions, discriminatory activation functions, and a universal approximation theorem. In the same year, Korolev \cite{Korolev22TwoLayerNeuralNetworks} established both inverse and direct approximation theorems with Monte Carlo rates for an extended class of Barron spaces of vector-valued functions (also referred to as vector-valued $\mathcal{F}_1$ functions) when using shallow neural networks. These results were derived using the \Frechet{} metric as the error measure.
Similarly, the purpose of the present paper is to provide theoretical statements about asymptotic approximation rates. However, we adopt different techniques and focus on operators with symbols, which reduce the complexity to function approximation, thereby streamlining the analysis.

The theory of approximating functions by neural networks has already been studied for several decades \cite{Hornik91ApproximationCapabilitiesMultilayer,Cybenko89ApproximationSuperpositionsSigmoidal,Mhaskar92ApproximationSuperpositionSigmoidal} and is still a very active field. 
Once a universal approximation theorem is established, a crucial next step is to derive asymptotic bounds on the approximation error, offering deeper insights into the accuracy of the approximation. 
The study of approximation rates for functions is already a well-developed field \cite{Barron93UniversalApproximationBounds, Leshno93MultilayerFeedforwardNetworks, Makovoz98UniformApproximationNeural}. However, it remains an active area of research, particularly for data-driven surrogate models like neural networks, where questions about novel settings, optimality, and the practical realization of these rates remain open.

Recent advancements in the theory and applications of neural networks have led to scenarios where the goal of machine learning is to approximate functions whose regularity is known beforehand \cite{Cybenko89ApproximationSuperpositionsSigmoidal, Yarotsky18OptimalApproximationContinuous, Abdeljawad22Approximationsdeepneural,Abdeljawad22UniformApproximationQuadratic,Montanelli21DeepReLUNetworks,Yang24Nearoptimaldeepneural,Siegel23Optimalapproximationrates,Han18Solvinghighdimensionalpartiala,Bolcskei19OptimalApproximationSparsely,Lu21DeepNetworkApproximation}.
A very fruitful field in this regard is the study of shallow neural networks which serve as foundational building blocks, offering insights into key properties such as universality, approximation rates, and expressivity \cite{Barron93UniversalApproximationBounds, Siegel20ApproximationRatesNeural, Ma22UniformApproximationRates, Li23TwolayerNetworkstext,Lu22PrioriGeneralizationError,Luo20TwoLayerNeuralNetworks,E21KolmogorovWidthDecay,Savarese19HowInfiniteWidth, Abdeljawad22IntegralRepresentationsShallow}. Their simplicity makes them ideal for theoretical analysis and generalization, providing a baseline for understanding more complex behaviors deep architectures.
A prominent concept emerging from this line of research is the study of spectral Barron spaces, which encapsulate functions that can be efficiently approximated by two-layers neural networks under spectral constraints. These spaces provide a mathematically rigorous framework for analyzing the expressivity of networks and their capacity to approximate target functions with controlled regularity and complexity \cite{Barron93UniversalApproximationBounds, Abdeljawad23SpaceTimeApproximationShallow, Siegel20ApproximationRatesNeural, Siegel22HighOrderApproximationRates}.
The initial results of \citeauthor{Barron93UniversalApproximationBounds, Hornik91ApproximationCapabilitiesMultilayer, Makovoz98UniformApproximationNeural} were approximation rates in terms of the $L^2$ and $L^\infty$ norms.
Recently, these results have been extended to the Hilbert-Sobolev Norm $H^m$ for $m\in\nn{}$, the rate has been refined from $N^{-1/2}$ to rates in the order of $N^{-\frac{1}{2}-\frac{s-m}{d}}$ \cite[Theorem 1]{Siegel22HighOrderApproximationRates} for target functions with finite regularity $s$, $m$ is the derivative-degree of the Sobolev-error, and $d$ is the input dimension of the target function.
The case with infinite regularity of the target function was considered in \cite[Theorem 2]{Siegel22HighOrderApproximationRates}, where it was shown that this leads to a (sub-)exponential decay of the error.

In this work, we generalize the approximation setting from normed spaces and metric spaces to \Frechet{} spaces, which is to the best of our knowledge the first such extension.
We provide two high-level results which connect approximation rates for the individual semi-norms to a sufficient width of shallow networks such that a certain error measured in a \Frechet{} metric $d_\set{V}$ is achieved.
That is, we provide the inverse of the approximation rate in closed form.

In the first result we assume that the sequence of semi-norms is monotonically growing when measuring any fixed target function.
This condition is rather mild; if the sequence $\{p_i\}_{i\in\nn{}}$ does not grow, then $\{q_i\}_{i\in\nn{}}$ with $q_i:=\max_{j\leq i} p_j$ does.
However, the result requires full knowledge of the asymptotic approximation rates in every semi-norm.
We will show that these assumptions can be fulfilled, when considering a sequence of Sobolev norms of increasing order together with the exponential spectral Barron space as class of target functions and cosine activated shallow neural networks as approximation class.
We furthermore show that this type of Barron space is embedded in a class of Fourier multipliers, which means that certain sub-classes of the H\"{o}rmander symbols are feasible for this type of approximation.
We furthermore show that our result can also be applied to the Gelfand-Shilov class of test functions.

In the second result we impose an upper bound on the growth of the semi-norms, which is a strong condition, but for this result we only require knowledge of the asymptotic approximation rate for the first semi-norm.
As an example we will again consider the sequence of Sobolev norms of increasing order.
However, this time we will show that this sequence does not fulfill the bounded-growth-condition when considering the exponential spectral Barron space as target class.
Instead we explore different classes of functions such as self-weighted Gevrey symbols and a class of Paley-Wiener functions with a mild smoothness constraint in the Fourier domain.
For the latter, we then show that there is an approximation class for which the asymptotic approximation rate in the $L^2(\rr{d})$-norm is the Monte-Carlo rate of order $1/2$.
That is, the approximation is possible without curse of dimensionality, even though the approximation is over the full domain.
Furthermore, this immediately leads to an approximation result in the \Frechet{} metric.

With our work, we lay the ground for approximation of pseudodifferential operators in their natural \Frechet{} metric.
Note that in the examples that are based on the exponential spectral Barron space and the Barron-Bandlimited functions we consider symbols which are independent of the first variable, that is, we can choose $\rho=1$ and by using the sequences of $H^\ell(\set{U})$ Sobolev norms we allow every $\delta\in[0,1)$.
Note however, that in the first result we have to choose $\set{U}$ as a bounded set which allows us to embed the class of target functions in a symbol space with bounded domain.
The Sobolev norm $H^\ell(\set{U})$ can then be seen as an $L^2$ variant and combination of all $p_{\alpha,\beta}$ with $\abs{\beta}\leq \ell$ over a bounded domain.
In the second result we can indeed allow an unbounded domain, however, the necessary $L^1$ constraint in order to impose the bandlimitedness excludes symbols of order $m>0$.
We leave the extension to the full class of H\"{o}rmander symbols open for future work.

The remainder of the paper is structured as follows.
As a first step, in \cref{sec:prelim}, we cover some preliminaries about function spaces and approximation results.
In \cref{sec:monotonic_growth} we present the approximation result with the monotonic-growth-condition and the concrete examples for which this leads to an approximation rate in the \Frechet{} metric of Sobolev norms.
Finally, in \cref{sec:bounded_growth}, we present the approximation result with the bounded-growth-condition and study its application to Barron-Bandlimited functions.

\section{Preliminaries}\label{sec:prelim}

In this work, for $L^1$ integrable functions, we use the following convention for the Fourier transform and its inverse:
$$
(\mathscr{F} u)(\xi)=\hat{u}(\xi)=\frac{1}{(2 \pi)^{d / 2}} \int_{\mathbb{R}^d} u(x) e^{-i x\cdot \xi} d x,\quad 
\left(\mathscr{F}^{-1} u\right)(x)=\frac{1}{(2 \pi)^{d / 2}} \int_{\mathbb{R}^d} u(\xi) e^{ix\cdot \xi} d \xi ,
$$
where $x\cdot \xi$ denotes the standard scalar product in $\rr{d}$.
We also recall that a \Frechet{} space $\mathbf{F}$ is a complete metrizable locally convex space.
If $\left\{p_i\right\}_{i \in \mathbb{N}}$ is a sequence
of continuous semi-norms on $\mathbf{F}$, then a metric can be defined as follows
$$
d(f, g)=\sum _{\ell \in \mathbb{N}} \frac{1}{2^\ell} \frac{p_\ell(f-g)}{1+p_\ell(f-g)},
$$
for each $f, g \in \mathbf{F}$.

\subsection{Function Spaces}

The core focus of our work lies in the use of spectral Barron spaces with exponential weights as target function classes.
Here we begin by introducing the concept of weighted Lebesgue, Sobolev, Fourier-Lebesgue and Gelfand-Shilov spaces.

\begin{definition}[Weighted Function Spaces]
    Let $p\in[1,\infty]$, $d\in\nn{}$, $m\in\zzp{}$, $\set{U}\subseteq\rr{d}$, and $\omega$ be a weight function on $\rr{d}$.
    We define the weighted Lebesgue space, the weighted Sobolev space, and the weighted Fourier-Lebesgue space as
    \begin{align*}
        L^p(\omega;\set{U})
        &:=
        \{f:\set{U}\to\rr{}\;\text{s.t.}\;\omega f\in L^p(\set{U})\},\\
        \sobolev{m,p}[\omega][\set{U}]
        &:=
        \{f:\set{U}\to\rr{}\;\text{s.t.}\;\forall \alpha\in\zzp{d}\,\text{with}\,\abs{\alpha}_1\leq m\,\text{it holds}\,\partial^\alpha f\in L^p(\omega;\set{U})\},\\
        \intertext{and}
        \FL{p}{\omega;\set{U}}
        &:=
        \{f:\set{U}\to\rr{}\;\text{s.t.}\;\exists f_e\in L^1(\rr{d})\,\text{with}\,f_e|_\set{U}=f\,\text{such that}\,\hat{f}_e\in L^p(\omega;\set{U})\},
    \end{align*}
    respectively.
    These spaces are normed spaces equipped with the norms
    \begin{align*}
        \norm{f}_{L^p(\omega;\set{U})}&:=\left(\int_\set{U}\abs{\omega(x)f(x)}^p\dee x\right)^{\frac{1}{p}},\\
        \norm{f}_{\sobolev{m,p}[\omega][\set{U}]}&:=\left(\sum\nolimits_{\abs{\alpha}\leq m}\norm{\omega\partial^\alpha f}{L^p(\set{U})}^p\right)^{\frac{1}{p}},\\
        \intertext{and}
        \norm{f}_{ \FL{p}{\omega;\set{U}}}&:=\inf_{\substack{f_e\in L^1\\\left.f_e\right|_\set{U}=f}}
        \norm{\omega\hat{f}_e}{L^{p}(\rr{d})},
    \end{align*}
    respectively.
    For $p=\infty$, we make the obvious modifications for $L^p(\rr{d})$ and replace the summation in the weighted Sobolev norm by a maximum.
\end{definition}
We omit the domain or the weight from the notation if $\set{U}= \rr{d}$or $\omega$ is a constant, respectively.
Note that if $\set{U}=\rr{d}$, the infimum in the Fourier-Lebesgue norm is over a single element (of equivalence classes), which therefore results in the implicitly requirement that $f\in L^1(\rr{d})$.

The original result of Barron space  $\mathcal{B}^1(\rr{d})$ was introduced in  \cite{Barron93UniversalApproximationBounds} consists of functions defined on $\rr{d}$ such that $\xi\cdot\hat{f}(\xi)$ is integrable. The fact that functions within these class can be approximated by shallow neural networks  without suffering from the curse of dimensionality has sparked significant research and numerous generalizations over the past decades.
For instance, in \cite{E22BarronSpaceFlowInduced} Barron spaces $\mathcal{B}_p$ are defined as the set of continuous functions that can be represented by 
\begin{equation}\label{eq:E_Barron_Def}
f(x)=\int_{\rr{}\times\rr{d}\times\rr{}} a \sigma\left(\boldsymbol{b}^T x+c\right) \rho(d a, d \boldsymbol{b}, d c), 
\end{equation}
with finite Barron norm:
$$
\|f\|_{\mathcal{B}_p}=\inf _\rho\left(\mathbb{E}_\rho\left[|a|^p\left(\|\boldsymbol{b}\|_1+|c|\right)^p\right]\right)^{1 / p},
$$
where the infimum is taken over all $\rho$ for which holds \eqref{eq:E_Barron_Def} for all functions $f$. This space can be viewed as a generalization of the shallow neural network representation of a function, extending it to an integral representation with respect to an underlying probability measure on the parameter space.
Furthermore, in \cite{Siegel22SharpBoundsApproximation, Siegel20ApproximationRatesNeural}, a spectral Barron space is defined as a Fourier-Lebesgue space with polynomial weights. That is, for $\Omega\subseteq \rr{d}$
$$
\|f\|_{\mathscr{B}_s(\Omega)}:=\inf _{f_e \mid_{\Omega}=f} \int_{\mathbb{R}^d}|\hat{f}(\xi)|(1+|\xi|)^s \mathrm{~d} \xi
$$
and the corresponding spectral Barron space is defined as 
$$
\mathscr{B}_s(\Omega) := \left\{f:\Omega\to \rr{}\text{ s.t. } \|f\|_{\mathscr{B}_s(\Omega)}<\infty\right\}.
$$
Note that, this kind of spaces has been already proposed by H\"ormander \cite{Hormander64LinearPartialDifferential}, where they are defined for $1 \leq p \leq \infty$ and  $s \geq 0$ as follows:
\[
\mathscr{F} L_p^s\left(\mathbb{R}^d\right):=\left\{f \in \mathscr{S}^{\prime}(\rr{d}) \text{ s.t. }  \left(1+|\xi|^s\right) \widehat{f}(\xi) \in L^p(\rr{d})\right\}.
\]
These spaces have been extensively studied by the microlocal analysis and pseudodifferential operator communities. A similar version has been also investigated in our previous papers \cite{Abdeljawad23SpaceTimeApproximationShallow, Abdeljawad24WeightedApproximationBarron}. Instead in \cite{Siegel22HighOrderApproximationRates} authors investigated the first Barron space with exponential  weights. These spaces will be also used in the current work.
For the exponential spectral Barron space we adopt the notation from \cite{Siegel22HighOrderApproximationRates}.
\begin{definition}[Exponential Spectral Barron Space]
    Let $\set{U}\subset\rr{d}$ be a bounded domain and let $\omega(\xi)=e^{c\abs{\xi}^\beta}$ with $\beta,c>0$.
    Then the exponential spectral Barron space over $\set{U}$ is defined as
    \begin{align*}
        \barron{\beta,c}[\set{U}]
        :=\left\{f:\set{U}\to\rr{}:\norm{f}{\barron{\beta,c}[\set{U}]}<\infty\right\},
    \end{align*}
    with the corresponding norm
    \begin{align*}
        \norm{f}[\barron{\beta,c}[\set{U}]]
        :=\inf_{\substack{f_e\in L^1\\\left.f_e\right|_\set{U}=f}}
        \norm{\omega \F\{f_e\}}[L^1(\rr{d})].
    \end{align*}
\end{definition}

\subsection{Multipliers and Gelfand-Shilov Spaces} 

The concept of a multiplier often arises in various contexts, such as in the study of operators, harmonic analysis, and signal processing.

\begin{definition}[Smooth Multipliers]\label{def:symbol}
Let $d\in \nn{}$,
$U\subseteq \rr{d}$,
$p\in [1,\infty]$
and
$\omega$ is a weight function defined on 
$\rr {d}$,
then the class of symbols
$\mathbf{S}^{p}_\omega( U)$
(where  $\omega$ characterizes the growth order
and  $ p$  indicates the order of integrability)
consists of all symbols
$a \in C^{\infty}\left(U\right)$
such that
\begin{equation}\label{eq:Lebesgue_symbol_norm}
\left(\int_{ U}
\left|\omega(x)\partial_x^\alpha a(x)\right|
^{p} dx\right)^{1/p} <\infty
\text{ for any }
 \alpha, \beta\in \zzp{d}.
\end{equation}
For $a\in {\mathbf{S}^{p}_\omega(U)}$,
we define the semi-norms
$\| a \|_{\mathbf{S}^{p}_{\omega, \ell}( U)}$ by
\begin{equation}\label{eq:symbol_norm}
\|a\|_{\mathbf{S}^{p}_{\omega, \ell}( U)}
:=
\max_{ |\alpha| \leq \ell }\;
\left(\int_{ U}
\left|\omega(x) \partial_x^\alpha a(x)\right|^p
dx\right)^{1/p},
\end{equation}
for any $\ell \in \zzp{}$,
if $p=\infty$, then we write ${\mathbf{S}_\omega(U)}$
instead of ${\mathbf{S}^{\infty}_\omega(U)}$
and we define
$\|a\|_{\mathbf{S}_{\omega,\ell}(U)}$
as follows
\begin{equation}\label{eq:sup_sup_symb_norm}
\|a\|_{\mathbf{S}_{\omega, \ell}( U)}
=
\max_{|\alpha|\leq \ell}
\;\sup_{x\in U}\;
\left|\omega(x)\partial_x^\alpha  a(x)\right|,
\end{equation}
for any $\ell\in \zzp{}$.
\end{definition}
The quantities \cref{eq:symbol_norm} define a Fr\'echet topology of 
${\mathbf{S}^{p}_\omega(U)}$.
It is obvious that our symbol class corresponds to the so-called
\emph{Gelfand-Shilov class} $S_s(\rr{d})$ if
$\omega(x) = e^{r|x|^{1/s}}$ for some $r >0$
in \cref{eq:sup_sup_symb_norm} for all $\ell\in \zzp{}$
cf.,
\cite{Chung96CharacterizationsGelfandShilovSpaces,
Gelfand16SpacesFundamentalGeneralized,
VanEijndhoven87FunctionalAnalyticCharacterizations}. 
One can also remark that the single block variant of Bochner-Sobolev spaces
defined in 
cf., \cite{Abdeljawad23SpaceTimeApproximationShallow}
is the limited version of the space ${\mathbf{S}^{p}_\omega(U)}$ 
of functions satisfying \cref{eq:symbol_norm}.
In the sense that, in \cref{eq:symbol_norm}
we allow the order of derivatives $\alpha$ to be unbounded.

Recently,  \emph{Gelfand-Shilov spaces} and their generalizations have been widely studied and used in many applications to partial differential equations, see  \cite{Abdeljawad20LiftingsUltramodulationSpaces,AriasJunior22CauchyProblem3evolution, Ascanelli19SchrodingertypeEquationsGelfandShilov,Carypis17PropagationExponentialPhase}.
Here we recall the definition of Gelfand-Shilov spaces and give some characterizations of these spaces, more details can be found in  \cite{Abdeljawad19PseudoDifferentialCalculusAnisotropic, Abdeljawad20LiftingsUltramodulationSpaces}.

\begin{definition}[Gelfand-Shilov Spaces]
Let $s,\sigma>0$ then the Gelfand-Shilov space $S_s^\sigma(\rr{d})$ consists of all smooth functions $u\in C^\infty(\rr{d})$ for which there is an $h>0$ such that
\begin{equation}\label{gfseminorm}
\norm{u}[{S _{s;h}^{\sigma}}]
\equiv
\sup_{\alpha\in \zzp{d} ,\beta \in\zzp{d}}\;
\sup_{x\in \rr {d}}
\frac {|x^{\alpha}\partial_x ^{\beta}
u(x)|}{h^{|\alpha+\beta|}
\alpha !^{\frac{1}{s}}\,
\beta !^{\frac{1}{\sigma}}}<\infty,
\end{equation}
\end{definition}

Obviously $S _{s;h}^{\sigma}(\rr{d})$ is a Banach space which increases as
$h, s $ and $\sigma$ increase,
and is contained in the Schwartz space $\mathscr{S}(\rr{d})$. 

The Gelfand-Shilov spaces $S_s^\sigma\left(\rr{d}\right)$ and $\Sigma_s^\sigma\left(\rr{d}\right)$ are the inductive and projective limits, respectively, of $S_{s, h}^\sigma\left(\rr{d}\right)$ with respect to $h$. This implies that
\begin{equation}\label{GSspacecond1}
S_s^\sigma\left(\rr{d}\right)=\bigcup_{h>0} S_{s, h}^\sigma\left(\rr{d}\right) \quad \text { and } \quad \Sigma_s^\sigma\left(\rr{d}\right)=\bigcap_{h>0} S_{s, h}^\sigma\left(\rr{d}\right),
\end{equation}
and that the topology for $S_s^\sigma\left(\rr{d}\right)$ is the strongest possible one such that each inclusion map from $S_{s, h}^\sigma\left(\rr{d}\right)$ to $S_s^\sigma\left(\rr{d}\right)$ is continuous, for every choice 
of $h>0$.

Next proposition shows different characterizations of the Gelfand-Shilov spaces, see 
\cite{Chung96CharacterizationsGelfandShilovSpaces,
VanEijndhoven87FunctionalAnalyticCharacterizations,
Gelfand16SpacesFundamentalGeneralized} for more details.

\begin{proposition}\label{prop:S_def_equiv_fourier}
Taking $s,\sigma > 0$, the following
conditions are equivalent.
\begin{propenum}
\item\label{it:GS:I} $u\in {S} _{s }^{\sigma }(\rr {d})$\quad
($u\in \Sigma _{s }^{\sigma}(\rr {d})$);

\item\label{it:GS:II} for some (for every) $r>0$  it holds
\begin{equation*}
\displaystyle{|u(x)|\lesssim e^{-r|x|^{s} }}
\quad \text{and}\quad 
\displaystyle{|\hat u(\xi )|\lesssim
e^{-r|\xi |^{\sigma} }},\quad x,\xi\in \rr{d}
\end{equation*}
\item\label{it:GS:III} for some (every) $h>0$ it holds

$$
|x^\alpha f(x)| \lesssim h^{|\alpha|} \alpha!^s \quad \text { and } \quad |\xi^\beta \widehat{f}(\xi)| \lesssim h^{|\beta|} \beta!^\sigma, \quad x,\xi\in \rr{d}, \alpha, \beta \in \mathbf{N}^d
$$

\item\label{it:GS:IV} for some (every) $h>0$ it holds

$$
|x^\alpha f(x)| \lesssim h^{|\alpha|} \alpha!^s \quad \text { and } \quad|\partial^\beta f(x)| \lesssim h^{|\beta|} \beta!^\sigma, \quad x\in \rr{d}, \alpha, \beta \in \mathbf{N}^d.
$$
\end{propenum}
\end{proposition}
Note that $S_s$ is characterized by the satisfaction of the left-hand side conditions in either \cref{it:GS:II,it:GS:III,it:GS:IV}.
Similarly, $S^\sigma$ is characterized by the satisfaction of the right-hand side conditions in any of \cref{it:GS:II,it:GS:III,it:GS:IV}.
It is worth mentioning that the Fourier transform extends uniquely to homeomorphisms from $\left(S_s^\sigma\right)^{\prime}\left(\rr{d}\right)$ to $\left(S_\sigma^s\right)^{\prime}\left(\rr{d}\right)$, and from $\left(\Sigma_s^\sigma\right)^{\prime}\left(\rr{d}\right)$ to $\left(\Sigma_\sigma^s\right)^{\prime}\left(\rr{d}\right)$. Furthermore, it restricts to homeomorphisms from $S_s^\sigma\left(\rr{d}\right)$ to $S_\sigma^s\left(\rr{d}\right)$, and from $\Sigma_s^\sigma\left(\rr{d}\right)$ to $\Sigma_\sigma^s\left(\rr{d}\right)$.

\subsection{Approximation with Shallow Networks}

Let $\mathbb{D}$ be a uniformly bounded dictionary in a Banach space $X$, that is, $\mathbb{D}$  is a subset of $X$ such that 
\[
\sup_{u\in\mathbb{D}} \|u\|_X<\infty.
\]
Nonlinear dictionary approximation focuses on approximating a function  using finite dictionary expansions, meaning approximations are drawn from a specific nonlinear set of the form:
\[
\Sigma_N(\mathbb{D}):= \left\{ \sum_{k=1}^N a_ku_k \text{ s.t. }u_k\in \mathbb{D}\right\}.
\]
It is often crucial to impose control over the coefficients that appear in the expansion, that is,
\[
\Sigma_{N,M}^p(\mathbb{D}):= \left\{ \sum_{k=1}^N a_ku_k \text{ s.t. }u_k\in \mathbb{D} \text{ and } \|a\|_{\ell^p}\leq M\right\},
\]
note that if $p=1$ we simply denote $\Sigma_{N,M}(\mathbb{D})$ instead of $\Sigma_{N,M}^1(\mathbb{D})$.

A special case that we will consider is the class of shallow networks with (complex) cosine activation function as in \cite[eq. (16)]{Siegel22HighOrderApproximationRates}
\begin{align}\label{eq:ExpSpecBarronSpace}
    \Sigma_{N,M}
    =
    \left\{
        \sum_{n=1}^Na_ne^{i2\pi \theta_n x}:\theta_n\in\rr{}, a_n\in\cc{},\sum_{n=1}^N\abs{a_n}\leq M
    \right\}.
\end{align}
When $M=\infty$ we denote $\Sigma_{N,\infty}$ by $\Sigma_N$.
This class exhibits an exponential approximation rate when approximating targets from the exponentially weighted spectral Barron space.

\begin{proposition}{{Approximation of Exponential Spectral Barron Space \cite[Theorem 2]{Siegel22HighOrderApproximationRates}}}
    \label{prop:siegel_exp:explicit}
    Let $\set{U}=[0,1]^d$, $0<\beta<1$, and $c>0$.
    Then for any $m\geq 0$, there exists a $c^\prime>0$ such that for $f\in \barron{\beta,c}[\set{U}]$ and $M\lesssim \norm{f}[{\barron{\beta,c}[\set{U}]}]$ we have
    \begin{align*}
        \inf_{f_N\in\Sigma_{N,M}}
        \norm{f-f_N}[{H^{m}(\set{U})}]
        \lesssim
        \norm{f}[{\barron{\beta,c}[\set{U}]}]
        e^{-c^\prime N^\frac{\beta}{d}}.
    \end{align*}
\end{proposition}

\section{\Frechet{}-Approximation with Semi-Norms of Monotonic Growth}
\label{sec:monotonic_growth}

As motivated in \cref{sec:intro}, we aim to approximate functions in the multipliers class and measure convergence in the corresponding topology that is induced by a separable sequence of semi-norms.
To formulate this problem in a general way, we consider the space $\targetSpace{}$ of target functions to be a vector space on which $\{\seminorm{\cdot}[\ell]\}_{\ell=0}^\infty$ is a separable sequence of semi-norms.
By \cite[Remark 1.38 (c)]{Rudin91FunctionalAnalysis} it is known that the topology that is induced by the sequence of semi-norms, is indeed the same as the the topology that is induced by
\begin{align*}
    d_\targetSpace(f):=\sum_{k\in\zzp{}}\frac{1}{2^k}\frac{\seminorm{f}[k]}{1+\seminorm{f}[k]}.
\end{align*}
By imposing a certain growth condition on the sequence of semi-norms and assuming knowledge of approximation results for each separate semi-norm, we can show the following approximation result for $d_\targetSpace$.

\begin{theorem}[Approximation with Monotonic Semi-Norm-Growth]
\label{thm:approximation_Frechet_metric:upwards}
    Let   
    $\targetSpace{}$ be a vector space and
    $\{\seminorm{\cdot}[\ell]\}_{\ell=0}^\infty$
    be a separable sequence of semi-norms on $\targetSpace{}$.
    Assume that there is a sequence $\{M_{\ell}\}_{\ell=0}^\infty$
    such that
    for all $f\in\targetSpace{}$ we have
    \begin{align}
    \label{eq:frechet:MlBound:upwards}
        \seminorm{f}[k]
        \leq
        M_{\ell} \seminorm{f}[\ell]
        \qquad\text{for every}\,k,\ell \in \zzp{}\,\text{with}\,k\leq \ell.
    \end{align}
    Furthermore, for all $\nNeuron\in\nn{}$ let $\Sigma_\nNeuron\subset \targetSpace{}$ be the approximation class of order $\nNeuron$
    such that for all $f\in\targetSpace{}$ and all $\ell\in\nn{}$ there is 
    a decreasing and bijective function $r_\ell:[1,\infty)\to(0,\bar{r}_{\ell}]$ (the rate) for some $\bar{r}_{\ell}>0$
    such that
    \begin{align}
    \label{eq:assumption_rate:upwards}
        \inf_{f_\nNeuron\in\Sigma_\nNeuron}\seminorm{f-f_\nNeuron}[\ell]
        \leq 
        C_f
        r_\ell(\nNeuron)
    \end{align}
    for some constant $C_f>0$ that may depend on the function $f$.
    Then, for any $\varepsilon>0$
    it is sufficient to choose
    \begin{align*}
        N
        \geq
        N_{\ell_\varepsilon}
        :=
        \left\lceil 
            r_{\ell_\varepsilon}^{-1}
            \left(
                \min\left\{r_{\ell_\varepsilon}(1),\frac{1}{2^{\ell_\varepsilon} C_fM_{\ell_\varepsilon}}\right\}
            \right)
        \right\rceil
        \qquad\text{with}\qquad
        \ell_\varepsilon:=\lceil-\log_2(\varepsilon)\rceil+1
    \end{align*}
    in order to achieve
    \begin{align*}
        \inf_{f_\nNeuron\in\Sigma_\nNeuron} d_{\targetSpace{}}(f-f_\nNeuron)
        =
        \inf_{f_\nNeuron\in\Sigma_\nNeuron}
            \sum_{\ell=0}^\infty\frac{1}{2^\ell}
            \frac{\seminorm{f-f_\nNeuron}[\ell]}
                {1+\seminorm{f-f_\nNeuron}[\ell]}
        < \varepsilon.
    \end{align*}
\end{theorem}
It is worth noting that \cref{eq:frechet:MlBound:upwards} is only a mild assumption, as it can be satisfied by various norms, including the Sobolev norms with increasing regularity and/or domain or $L^p$ norms with increasing domain.
The major difficulty in the proof of this theorem is that the infimum is jointly over all $\ell$ which means, that we cannot bound the error immediately by the approximation rate.
Similarly, pulling the infimum inside the sum would allow to apply the approximation result,
however, doing so would mean that the approximation possibly results in different solutions $f_N$ for each $\ell$.
Instead, our proof is based on a standard trick that for handling this type of \Frechet{} metrics.
Namely, we split the series at some finite $\ell$, which lets us bound the tail of the sum by a constant that depends only on $\ell$ and for the leading $\ell$ terms we provide an upper bound in terms of an approximation rate.
Overall, we can then derive the necessary order of the approximation class.
Note that, if $r$ is the Monte Carlo rate i.e., $r(N) = N^{-\frac{1}{2}}$, then it is straightforward that $\bar{r}=1$. Conversely, if $r$ is an exponential rate function $r(N) = e^{-N^\beta}$ for $\beta\in (0,1)$ then $\bar{r} = \frac{1}{e}$.
\begin{proof}
For a fixed $\ell\in\zzp{}$, first, we split $d_\targetSpace{}$ into a sum over the first $\ell$ elements and the remaining tail.
The finite sum can be bound by its last element using \cref{eq:frechet:MlBound:upwards} and using the fact that $t\mapsto t/(1+t)$ is strictly monotonically increasing.
The remaining sums can then be evaluated in closed form
\begin{equation}
\label{eq:truncated_bound}
\begin{split}
    d_{\targetSpace{}}(f-f_\nNeuron)
    &=
    \sum_{k=0}^\infty
        \frac{1}{2^k}
        \frac{\seminorm{f-f_\nNeuron}[k]}{1+\seminorm{f-f_\nNeuron}[k]}
    \\&
    =
    \sum_{k=0}^{\ell}
        \frac{1}{2^k}
        \frac{\seminorm{f-f_\nNeuron}[k]}{1+\seminorm{f-f_\nNeuron}[k]}
    +
    \sum_{k=\ell+1}^\infty
        \frac{1}{2^k}
        \frac{\seminorm{f-f_\nNeuron}[k]}{1+\seminorm{f-f_\nNeuron}[k]}
    \\&
    \leq
    \frac{M_{\ell}\seminorm{f-f_\nNeuron}[\ell]}{1+M_{\ell}\seminorm{f-f_\nNeuron}[\ell]}
    \sum_{k=0}^{\ell}
        \frac{1}{2^k}
    +
    \sum_{k=\ell+1}^\infty
        \frac{1}{2^k}
    \\&
    \leq
    (2-2^{-\ell})
    \frac{M_{\ell}\seminorm{f-f_\nNeuron}[\ell]}{1+M_{\ell}\seminorm{f-f_\nNeuron}[\ell]}
    +
    2^{-\ell}.
\end{split}
\end{equation}
As a next step we apply the infimum and due to the monotonicity and the continuity of $t\to t/(1+t)$ for $t\geq 0$, we can transition to the approximation rate $r_\ell$.
\begin{align*}
    \inf_{f_\nNeuron\in\Sigma_\nNeuron}
    d_{\targetSpace{}}(f-f_\nNeuron)
    &
    \leq
    \inf_{f_\nNeuron\in\Sigma_\nNeuron}
    (2-2^{-\ell})
    \frac{M_{\ell}\seminorm{f-f_\nNeuron}[\ell]}{1+M_{\ell}\seminorm{f-f_\nNeuron}[\ell]}
    +
    2^{-\ell}
    \\&
    \leq
    \frac{C_fM_{\ell}r_{\ell}(\nNeuron)}{1+C_fM_{\ell}r_{\ell}(\nNeuron)}
    (2-2^{-\ell}) 
    +
    2^{-\ell}
    \\&\leq
    C_fM_{\ell}r_{\ell}(\nNeuron)
    +
    2^{-\ell}.
\end{align*}
In order to determine a sufficient condition on $\nNeuron$ for achieving an error of at most $\varepsilon>0$, we now define
\begin{align*}
    \set{K}_\ell=\{N\in\nn{}: C_f M_\ell r_\ell(\nNeuron)\leq 2^{-\ell}\}.
\end{align*}
With $r_\ell$ being bijective and decreasing and $2^\ell C_fM_\ell\geq 1$ we can express the minimum of $\set{K}_\ell$ in closed form as
\begin{align*}
    N_\ell
    =
    \min\,\set{K}_\ell
    =
    \left\lceil r_\ell^{-1}\left(\min\left\{r_\ell(1),\frac{1}{2^\ell C_fM_\ell}\right\}\right)\right\rceil
    .
\end{align*}
Thus, for $\nNeuron\geq N_\ell$ we have
\begin{align*}
    \inf_{f_\nNeuron\in\Sigma_\nNeuron}
    d_{\targetSpace{}}(f-f_\nNeuron)
    &
    \leq
    C_fM_\ell r_\ell(N)+2^{-\ell}
    \leq
    C_fM_\ell r_\ell(N_\ell)+2^{-\ell}
    \leq
    2^{-\ell+1}.
\end{align*}
By choosing 
$\ell_\varepsilon=\lceil-\log_2(\varepsilon)\rceil+1$
we achieve an error
\begin{align*}
    \inf_{f_\nNeuron\in\Sigma_\nNeuron}d_{\targetSpace{}}(f-f_\nNeuron)
    \leq C_fM_{\ell_\varepsilon} r_{\ell_\varepsilon}(N_{\ell_\varepsilon})+2^{-{\ell_\varepsilon}}
    \leq2^{-\lceil-\log_2(\varepsilon)\rceil}
    <\varepsilon
\end{align*}
for $N\geq N_{\ell_\varepsilon}$.
\end{proof}
The central element of this theorem is the sequence of semi-norms that induces the topology of the approximation space, as well as the necessity of understanding the corresponding approximation rates.
Trivially, this assumption can be fulfilled by considering a constant sequences $\seminorm{\cdot}[k]=\norm{\cdot}{}$ for some norm $\norm{\cdot}{}$ (note that it is necessary to consider a norm in order to fulfill the separability).
This allows to consider $M_\ell= 1$ for all $\ell\in\zzp{}$, and therefore
\begin{align*}
    N_\ell=\left\lceil r^{-1}\left(2^{-\lceil-\log_2(\varepsilon)\rceil-1}\right)\right\rceil.
\end{align*}

\begin{remark}
    Note that the fact that pseudodifferential operators can be written as linear operators with symbols \cite{Abdeljawad19PseudoDifferentialCalculusAnisotropic}, using the previous theorem one can extend the result to an approximation bound for operators, in the sense that, $\operatorname{Op}[f]$ is a pseudodifferential operator with symbol $f$ defined as follows:
    \[
        \operatorname{Op}[f]u(x) := \mathcal{F}^{-1}\left( f(\xi)\hat{u}(\xi)\right)(x).
    \]
    The previous identity, combined with the result of the current section, yields an expressivity result for neural networks when approximating pseudodifferential operators under more restrictive assumptions on the metric.
\end{remark}

\subsection{Approximation of Exponential Spectral Barron Space}

A more interesting setting arises with the exponential spectral Barron space.
In this case, we consider the sequence of Sobolev norms, i.e., $\seminorm{\cdot}[k]=\norm{\cdot}[H^k(\set{U})]$, for which the growth condition is fulfilled with $M_\ell = 1$ for all $\ell\in\zzp{}$.
By restricting to functions in the exponential spectral Barron space ensures that the functions belong to $\bigcap_{\ell\in\zzp{}}H^\ell(\set{U})$.
Furthermore,  by \cite[Theorem 2]{Siegel22HighOrderApproximationRates} we obtain  an approximation rate $r_\ell$ for all $\ell\in\zzp{}$.
Overall, this leads to the following approximation result.
\begin{corollary}[Approximation of Exponential Spectral Barron Space]
\label{cor:sbs_upwards}
    Let $\set{U}=[0,1]^d$, $c>0$, $\beta \in (0,1)$ and $\Sigma_{N,M}$ be the class of shallow networks with $\nNeuron$ neurons and cosine activation function where $\nNeuron\in \nn{}$ and $M>0$ (see \cref{eq:ExpSpecBarronSpace}).
    Then there exists sequences $\{c_\ell\}_{\ell=0}^\infty\subset\rr{}_{>0}$ and $\{C_\ell\}_{\ell=0}^\infty\subset\rr{}_{>0}$ such that for any $f\in\barron{\beta,c}[\set{U}]$ and $M\lesssim \norm{f}[\barron{\beta,c}[\set{U}]]$ we get that
    \begin{align*}
        \inf_{f_\nNeuron\in\Sigma_{N,M}}d_{\targetSpace{}}(f-f_\nNeuron)\leq \varepsilon
    \end{align*}
    for any $\epsilon>0$ and 
    \begin{align*}
        N
        \geq
        \max\left\{
        1,
        \left\lceil
            \left(
            \frac{1}{c_{\ell_\varepsilon}}
            \ln\left(2^{\ell_\varepsilon}
                {C_{\ell_\varepsilon}\norm{f}[\barron{\beta,c}[\set{U}]]}
            \right)
            \right)^{\frac{d}{\beta}}
        \right\rceil
        \right\}
        \qquad\text{with}\qquad
        \ell_\varepsilon:=\lceil-\log_2(\varepsilon)\rceil+1.
    \end{align*}
\end{corollary}
\begin{proof}
    The result of the corollary is a direct implication of \cref{thm:approximation_Frechet_metric:upwards} combined with
    \cite[Theorem 2]{Siegel22HighOrderApproximationRates} (see \cref{prop:siegel_exp:explicit})
    with the rate function for each $\ell\in\nn{}$ being 
    \begin{align*}
        r_\ell(N)=C_\ell e^{-c_\ell N^{\frac{\beta}{d}}}.
    \end{align*}
    That is, for a given $\ell$ the constants $c_\ell$ and $C_\ell$ correspond to the constant $c^\prime$ in the exponent \(e^{-c^\prime N^{\frac{\beta}{d}}}\) and the hidden constant in the rate of \cref{prop:siegel_exp:explicit}, respectively.
\end{proof}

In reference to our initial target of approximating functions in the multipliers class, we observe that this approximation result is indeed applicable to certain classes of symbols that are defined in \cref{def:symbol}.
\begin{lemma}[Embedding in Multipliers Class]
\label{lem:symbol_embedding}
    Let $\set{U} \subseteq \rr{d}$,
    $c,\beta>0$, and $f\in\barron{\beta,c}[\set{U}]$, then we get
    \begin{align*}
        \abs{\partial^\alpha f(x)}
        \leq
        \norm{f}[\barron{\beta,c}[\set{U}]]
        \left(\frac{1}{c\beta}\right)^{\frac{\abs{\alpha}}{\beta}}\left(\abs{\alpha}!\right)^\frac{1}{\beta}
    \end{align*}
    for all $\alpha\in\zzp{d}$ and $x\in\set{U}$.
    Thus, $\barron{\beta,c}[\set{U}]\hookrightarrow \hormander{1,\ell}{}[\set{U}]$ for any $\ell\in\zzp{}$.
\end{lemma}
\begin{proof}
    The definition of the spectral Barron space is based on the infimum over the $\FL{1}{\omega}$-norm of all possible extensions.
    It is not guaranteed, that the infimum is indeed attained and therefore we show that the claim holds in the limit of a sequence of extensions that converges to the infimum.
    By definition for any $f\in\barron{\beta,c}[\set{U}]$, $\omega (\xi)= e^{c |\xi|^\beta}$, such a sequence of extensions $\{f_{e,n}\}_{n\in\nn{}}\subset L^1(\rr{d})$ with $\left.f_{e,n}\right|_{\set{U}}=f$ exists such that
    \begin{align*}
        \lim_{n\to\infty}\norm{f_{e,n}}[\FL{1}{\omega}]
        =\inf_{\substack{f_e\in L^1\\\left.f_e= f\right|_\set{U}}}\norm{f_e}[\FL{1}{\omega}]
        =\norm{f}[\barron{\beta,c}[\set{U}]],
    \end{align*}
    and
    \begin{align*}
        \norm{f_{e,n}}[\FL{1}{\omega}]
        =\norm{\hat{f}_{e,n}}[L^1(\omega)]
        =\int_{\rr{d}}\omega(\xi)\abs{\hat{f}_{e,n}(\xi)}\dee \xi
        <\infty.
    \end{align*}
    For the bound on the partial derivatives, we now continue as follows:
    \begin{align*}
        \abs{\partial^\alpha f_{e,n}(x)}
        &=
        \abs*{\int_{\rr{d}}e^{ix\cdot\xi}\xi^\alpha\hat{f}_{e,n}(\xi)\dee \xi}
        \leq
        \norm*{\xi^{\alpha}\hat{f}_{e,n}(\cdot)}[L^1]
        \\&
        \leq
        \norm{f_{e,n}}[\FL{1}{\omega}]
        \norm*{\abs{\cdot}^{\alpha}\omega(\cdot)^{-1}}[L^\infty].
    \end{align*}
    Only the second term is dependent on $\alpha$ and we continue its evaluation
    \begin{align}
        \norm*{\abs{\cdot}^{\abs{\alpha}}\omega(\cdot)^{-1}}[L^\infty]
        &=
        \sup_{\xi\in\rr{d}}\abs{\xi}^{\abs{\alpha}}e^{-c\abs{\xi}^\beta}
        =
        \sup_{t\geq 0}t^{\abs{\alpha}}e^{-c t^\beta}
        \label{eq:weight_bound}
    \end{align}
    for which we obtain the supremum for $\bar{t}$ with
    \begin{align*}
        0
        =
        \left.\frac{\dee}{\dee t}t^{\abs{\alpha}}e^{-c t^\beta}\right|_{t=\bar{t}}
        =
        \left.
        \abs{\alpha}t^{\abs{\alpha}-1}e^{-c t^\beta}
        -c\beta t^{\beta-1+\abs{\alpha}}e^{-c t^\beta}
        \right|_{t=\bar{t}}
    \end{align*}
    thus,
    \begin{align*}
        \bar{t}
        =
        \left(\frac{\abs{\alpha}}{c\beta}\right)^{\frac{1}{\beta}}.
    \end{align*}
    We insert this expression in \cref{eq:weight_bound} and apply the Stirling bound for the factorial to get
    \begin{align*}
        \sup_{t\geq 0}t^{\abs{\alpha}}e^{-c t^\beta}
        &=
        \left(\frac{\abs{\alpha}}{c \beta}\right)^{\frac{\abs{\alpha}}{\beta}}e^{-\frac{\abs{\alpha}}{\beta}}
        =
        \left(\frac{1}{c\beta}\right)^{\frac{\abs{\alpha}}{\beta}}\left(\left(\frac{\abs{\alpha}}{e}\right)^{\abs{\alpha}}\right)^\frac{1}{\beta}
        \\
        &\leq
        \left(\frac{1}{c\beta}\right)^{\frac{\abs{\alpha}}{\beta}}\left(\abs{\alpha}!\right)^\frac{1}{\beta}.
    \end{align*}
    Overall, we get for all $n\in\nn{}$ and $x\in\set{U}$
    \begin{align}\label{eq:GStype_ineq}
        \abs{\partial^\alpha f_{e,n}(x)}
        \leq
        \norm{f_{e,n}}[\FL{1}{\omega}]
        \left(\frac{1}{c\beta}\right)^{\frac{\abs{\alpha}}{\beta}}\left(\abs{\alpha}!\right)^\frac{1}{\beta}
    \end{align}
    and therefore the upper bound on $\abs{\partial^\alpha f(x)}$ follows by taking the limit $n\to\infty$.
\end{proof}

\subsection{Approximation of Gelfand-Shilov Space}

For functions in the exponential spectral Barron space $\barron{\beta,c}[\set{U}]$,
we now observe that any possible extension $f_e$ belongs to $L^1(\rr{d})$.
Therefore, 
\begin{align*}
\sup_{\xi\in\rr{d}}\hat{f}_e(\xi)<\infty
\end{align*}
and subsequently by multiplication with the exponential weight ($\beta,c >0$) we get the pointwise expression 
\begin{align*}
\abs{\hat{f}_e(\xi)}e^{c\abs{\xi}^\beta}<\infty\qquad\text{for all}\qquad\xi\in\rr{d}.\end{align*}
As $f\in\barron{\beta,c}[\set{U}]$, we may assume that $\norm{f_e}[\barron{\beta,c}]<\infty$.
Due to $\hat{f}_e\in L^{1}(e^{c\abs{\cdot}^\beta})$ we have
\begin{align*}
    \lim_{\abs{\xi}\to\infty}\abs{\hat{f}_{e}(\xi)}e^{c\abs{\xi}^\beta}=0,
\end{align*}
therefore, not only $\abs{\hat{f}_{e}(\xi)}\omega(\xi)<\infty$ but also $\sup_{\xi\in\rr{d}}\abs{\hat{f}_{e}(\xi)}e^{c\abs{\xi}^\beta}<\infty$.
Thus, we have that any extension $f_e$ of $f$ in the definition of the exponential spectral Barron space is a Gelfand-Shilov function, namely, $f_{e}\in S^\beta$.
Note, however, that this does not provide any information whether $f\in S^\beta$ because $f$ may not be defined outside its domain $\set{U}$.

Nevertheless, the question arises, whether the previous result can be extended to the Gelfand-Shilov spaces.
Indeed, we get that Gelfand-Shilov functions are members of certain parametrizations of the exponential spectral Barron space:
\begin{lemma}[Inclusion of Gelfand-Shilov Spaces]
    \label{lem:gs_embedding}
    Let $\set{U}\subset\rr{d}$ and $\beta>0$, then
    $f\in\GS{\beta}$ ($f\in\GS*{\beta}$) is in $\barron{\beta,h}[\set{U}]$ for some (for every) $h>0$.
\end{lemma}
\begin{proof}
    Let $f\in {S}^\beta(\rr{d})$ and consequently $f\in L^1(\rr{d})$, then there is an $h>0$ such that $f\in\FL{\infty}{e^{h\abs{\xi}^\beta}}$ and therefore, $f\in\FL{1}{e^{h^\prime\abs{\xi}^\beta}}$ for any $h^\prime$ with $0<h^\prime<h$, which concludes that $f\in\FL{1}{e^{h^\prime\abs{\xi}^\beta}}[\set{U}]$.
    For $f\in\Sigma^\beta(\rr{d})$ we can extend the argument analogously to all $h>0$ and get $f\in\FL{1}{e^{h^\prime\abs{\xi}^\beta}}[\set{U}]$ for any $h^\prime>0$.
\end{proof}
Note that in this embedding result, the choice of the parameter $h$ depends on the function $f$ for $f\in \GS{\beta}$, whereas it can be chosen freely for $f\in \GS*{\beta}$.
When extending the approximation results for the exponential spectral Barron space by means of this embedding result, the dependence of the parameter $h$ on the function $f$ directly reflects in the dependence of the sequences of constants in the rate function on $f$.
\begin{corollary}[Approximation of Gelfand-Shilov Space]
\label{cor:gs_upwards}
    Let $\set{U}=[0,1]^d$, $\beta \in (0,1)$ and $\Sigma_{N,M}$ be the class of shallow networks with $\nNeuron$ neurons and cosine activation function where $\nNeuron\in \nn{}$ and $M>0$.
    Then, for any $f\in S^\beta(\rr{d})$ ($f\in \Sigma^\beta(\rr{d})$)
    there exist sequences $\{c_\ell\}_{\ell=0}^\infty\subset\rr{}_{>0}$ and $\{C_\ell\}_{\ell=0}^\infty\subset\rr{}_{>0}$
    that are dependent on (independent of) $f$ such that
    for any $\epsilon>0$
    we get
    \begin{align*}
        \inf_{f_\nNeuron\in\Sigma_{N,M}}d_{\targetSpace{}}(f-f_\nNeuron)\leq \varepsilon
    \end{align*}
    when
    $M\lesssim \norm{f}[\barron{\beta,c}[\set{U}]]$ and
    \begin{align*}
        N
        \geq
        \max\left\{
        1,
        \left\lceil
            \left(
            \frac{1}{c_{\ell_\varepsilon}}
            \ln\left(2^{\ell_\varepsilon}
                {C_{\ell_\varepsilon}\norm{f}[\barron{\beta,c}[\set{U}]]}
            \right)
            \right)^{\frac{d}{\beta}}
        \right\rceil
        \right\}
        \quad\text{with}\quad
        \ell_\varepsilon:=\lceil-\log_2(\varepsilon)\rceil+1.
    \end{align*}
\end{corollary}

\begin{proof}
    The approximation result directly follows from \cref{cor:sbs_upwards} and \cref{lem:gs_embedding}.
\end{proof}

\section{\Frechet{}-Approximation with Semi-Norms of Bounded Growth}
\label{sec:bounded_growth}

For the general approximation result in \cref{thm:approximation_Frechet_metric:upwards} we require full knowledge of the approximation rate for all semi-norms in the sequence.
Obviously, this is a rather strong assumption and its relaxation requires further investigation.
In the present section, we first present an extension of \cref{thm:approximation_Frechet_metric:upwards}, in which we reduce the necessary knowledge of approximation rates to solely the first element of the sequence.
Secondly, we perform an analysis of cases, where this theorem with reduced necessary knowledge can be applied.
Namely, in \cref{sec:bounded_growth:sbs} we consider the exponential spectral Barron space and in \cref{sec:bounded_growth:bounded_derivatives} we consider other classes, for which the Sobolev norms fulfill the growth condition.

\begin{theorem}[Approximation with Bounded Semi-Norm-Growth]
\label{thm:approximation_Frechet_metric:downwards}
    Let 
    $\targetSpace{}$ be a vector space and
    $\{\seminorm{\cdot}[\ell]\}_{\ell=0}^\infty$
    be a separable sequence of semi-norms on $\targetSpace{}$.
    Assume that there is a non-decreasing sequence $\{M_{\ell}\}_{\ell=0}^\infty$
    such that
    for all $f\in\targetSpace{}$ we have
    \begin{align}
    \label{eq:frechet:MlBound:downwards}
        \seminorm{f}[\ell]
        \leq
        M_\ell \seminorm{f}[0]
        \qquad\text{for every}\qquad\ell \in \zzp{}.
    \end{align}
    Furthermore, for all $\nNeuron\in\nn{}$ let $\Sigma_\nNeuron\subset \targetSpace{}$ be the approximation class of order $\nNeuron$
    such that for all $f\in\targetSpace{}$
    there is 
    a decreasing bijective function $r:[1,\infty)\to(0,\bar{r}]$ (the rate) for some $\bar{r}>0$ such that
    \begin{align}
    \label{eq:assumption_rate:downwards}
        \inf_{f_\nNeuron\in\Sigma_\nNeuron}\seminorm{f-f_\nNeuron}[0]
        \leq C_fr(\nNeuron),
    \end{align}
    for some constant $C_f>0$ that may depend on $f$.
    Then, for any $\varepsilon>0$
    it is sufficient to choose
    \begin{align*}
        N
        \geq
        N_{\ell_\varepsilon}
        :=
        \left\lceil 
            r^{-1}
            \left(\min\left\{r(1),
                \frac{2^{-\ell_\varepsilon}}{C_fM_{\ell_\varepsilon}}\right\}
            \right)
        \right\rceil
        \qquad\text{with}\qquad
        \ell_\varepsilon:=\lceil-\log_2(\varepsilon)\rceil+1
    \end{align*}
    in order to achieve
    \begin{align*}
        \inf_{f_\nNeuron\in\Sigma_\nNeuron} d_{\targetSpace{}}(f-f_\nNeuron)
        =
        \inf_{f_\nNeuron\in\Sigma_\nNeuron}
            \sum_{\ell=0}^\infty\frac{1}{2^\ell}
            \frac{\seminorm{f-f_\nNeuron}[\ell]}
                {1+\seminorm{f-f_\nNeuron}[\ell]}
        < \varepsilon.
    \end{align*}
\end{theorem}

\begin{proof}
The proof can be concluded in a similar way as in the proof of
\cref{thm:approximation_Frechet_metric:upwards}.
The difference is that we now bound the finite part by its first element using the bounded-growth-condition
\begin{align*}
    \inf_{f_\nNeuron\in\Sigma_\nNeuron}
    d_{\targetSpace{}}(f-f_\nNeuron)
    &\leq
    \inf_{f_\nNeuron\in\Sigma_\nNeuron}
    \sum_{k=0}^{\ell}
        \frac{1}{2^k}
        \frac{M_{\ell}\seminorm{f-f_\nNeuron}[0]}{1+M_{\ell}\seminorm{f-f_\nNeuron}[0]}
    +
    \sum_{k=\ell+1}^\infty
        \frac{1}{2^k}
    \\&
    \leq
    \frac{M_{\ell}C_fr(\nNeuron)}{1+M_{\ell}C_fr(\nNeuron)}
    (2-2^{-\ell})
    +
    2^{-\ell}.
\end{align*}
and then develop the necessary condition for $\nNeuron$ analogously to the proof of \cref{thm:approximation_Frechet_metric:upwards}.
\end{proof}

\subsection{Limitations for Approximation of The Exponential Spectral Barron Space}
\label{sec:bounded_growth:sbs}

In our analysis of the exponential spectral Barron space $\barron{\beta,c}[\set{U}]$ we use again the Sobolev norms $\{\norm{\cdot}[H^\ell(\set{U})]\}_{\ell=0}^\infty$ as our sequence of semi-norms.
In order to apply the result from \cref{thm:approximation_Frechet_metric:downwards} the approximation class needs to be in the class of target functions, i.e, $\Sigma_{N,M}\subset\barron{\beta,c}[\set{U}]$, and the bound on the growth condition must be fulfilled, i.e.,
\begin{align*}
    \norm{f}[H^\ell(\set{U})]^2
    =
    \sum_{\abs{\alpha}\leq\ell}\norm{\partial^\alpha f}[L^2(\set{U})]^2
    \leq
    M_{\ell}^2
    \norm{f}[L^2(\set{U})]^2
\end{align*}
for some sequence $\{M_{\ell}\}_{\ell=0}^\infty$.
In the following, we will see that in any case, we cannot fulfill all conditions and therefore, we cannot apply \cref{thm:approximation_Frechet_metric:downwards} to the exponential spectral Barron space with Sobolev error.
The argument for that will be different for the two cases $\beta\geq 1$
and $\beta<1$.

For $\beta\geq 1$ we see that the approximation class is not in our space of target functions by using the analyticity that was obtained in \cref{lem:symbol_embedding}.
\begin{proposition}
    Let $\beta\geq 1$, then $\Sigma_{N,M}\not\subset\barron{\beta,c}[\set{U}]$ for all $M>0$ and  $N\in\nn{}$.
\end{proposition}
\begin{proof}
For $\beta\geq 1$ and $\set{U}\subseteq\rr{d}$, the pointwise bound in \cref{lem:symbol_embedding} implies that any Barron function $g\in\barron{\beta,c}[\set{U}]$ is real analytic in the interior of $\set{U}$ (see \cite{Hormander98AnalysisLinearPartial,Komatsu60CharacterizationRealAnalytic,Gelfand16SpacesFundamentalGeneralized}).
From \cref{lem:symbol_embedding} we know that for $f\in\barron{\beta,c}[\set{U}]$ all extensions $f_e$ are Barron functions over the full domain, i.e., $f_e\in\barron{\beta,c}=\barron{\beta,c}[\rr{d}]$.
Thus, $f$ is real analytic over $\set{U}$, $f_e$ is real analytic over $\rr{d}$, and the real analytic extension $f_e$ of $f$ is unique (see \cite[Theorem 16.11]{Rudin87RealComplexAnalysis}).
It follows for $f(x)=M\cos(\theta\cdot x+b)$ with $\theta\neq 0$ that the only real analytic extension is $f_e(x)=M\cos(\theta\cdot x+b)$.
This extension is not in $L^1(\rr{d})$ which means that $\norm{f}[\barron{\beta,c}[\set{U}]]=\infty$ as the infimum is over the empty set.
\end{proof}

For the case $\beta\in (0,1)$ we observe that the growth condition implies
\begin{align}
\label{eq:L2:MlBound}
    \norm{\partial^\alpha f}[L^2(\set{U})]
    \leq
    M_{\ell}
    \norm{f}[L^2(\set{U})]
\end{align}
for all multiindices $\alpha\in\zzp{d}$ with $\abs{\alpha}\leq \ell$,
which implies that the linear operator $\partial^\alpha$ has to be bounded.
It is a well known result the partial derivatives as an operator from $C^\infty(\set{U})$ endowed with the norm $\norm{\cdot}[L^2(\set{U})])$ is unbounded.
We now show that this holds as well if we restrict the space of functions to $\barron{\beta,c}[\set{U}]$, that is, we consider the operator $\partial^\alpha:(\barron{\beta,c}[\set{U}],\norm{\cdot}[L^2(\set{U})])\to(\barron{\beta,c}[\set{U}],\norm{\cdot}[L^2(\set{U})])$.
\begin{proposition}
    \label{prop:expSBS:counterexample}
    Let $\beta\in (0,1)$, $c>0$, $d\in \nn{}$, and $\set{U}=[0,2\pi]^d$.
    Then the operator $\partial^\alpha$ is unbounded with respect to the $L^2(\set{U})$-norm for functions in $\barron{\beta,c}[\set{U}]$.
\end{proposition}
In the proof, we use a classical counterexample for the unboundedness of $\partial^\alpha$ over $L^2(\set{U})$ for which we show that it still holds when limiting it to $\barron{\beta,c}[\set{U}]$.
\begin{proof}
We consider the case $d=1$; the extension to multiple dimensions can be carried out by taking the Cartesian product of this scalar function over the individual dimensions.
Similar to a text-book proof for the well know case $\partial^\alpha:(C^\infty(\set{U}),\norm{\cdot}[L^2(\set{U})])\to(C^\infty(\set{U}),\norm{\cdot}[L^2(\set{U})])$
we consider $\set{U}=[0,2\pi]$ and the sequence $\{f_n\}_{n\in\zzp{}}$, $f_n(x)=\frac{1}{\sqrt{\pi}}\sin(nx)$
for which we now show that for every $n\in\nn{}$ there is one specific extension $g_n$ of $f_n$ that is $g_n\in\barron{\beta,c}$.
Namely, we consider the smooth truncation
$$
g_n=f_n\cdot(\chF{U_\varepsilon}\conv\rho_\varepsilon)
\text{ with } U_\varepsilon=[-\varepsilon,2\pi+\varepsilon]\text{ and }\rho_\varepsilon=\frac{1}{\varepsilon}\phi(x/\varepsilon),
$$
where $\phi$ is a mollifier with $\phi\in\barron{\beta,c}$, $\phi(x)=0$ for $\abs{x}>\varepsilon$, and $\norm{\phi}{L^1}=1$.
By considering one specific extension, we derive an upper bound for the Barron norm, which is the infimum over all extensions.
To do so, we consider the Fourier transform of $\sin(nx)$ in the distributional sense, that is, $\mathscr{F}\{\sin(nx)\}(\xi)=\frac{1}{2in}(\delta_1(\xi/n)-i\delta_1(-\xi/n))$.
Then we get with $\omega(\xi)=e^{c\abs{\xi}^\beta}$
\begin{align*}
    \norm{f_n}[\barron{\beta,c}[\set{U}]]
    &\leq
    \norm{g_n}[\barron{\beta,c}[\set{U}]]
    =
    \norm{f_n(\chF{U_\varepsilon}\conv\rho_\varepsilon)}[\barron{\beta,c}[\set{U}]]
    \leq
    \norm{\hat{f}_n\conv(\widehat{\chF{U_\varepsilon}}\widehat{\rho_\varepsilon})}[L^1(\omega)]
    \\&
    =
    \int_{\rr{d}}\frac{1}{2n}\omega(\xi)
        \abs*{
            \widehat{\chF{\set{U}_\varepsilon}}{\left(\frac{\xi-n}{n}\right)}
            \widehat{\rho_\varepsilon}\left(\frac{\xi-n}{n}\right)
            -
            \widehat{\chF{\set{U}_\varepsilon}}{\left(\frac{\xi+n}{n}\right)}
            \widehat{\rho_\varepsilon}\left(\frac{\xi+n}{n}\right)}
    \dee \xi.
    \intertext{By means of the triangular inequality and seeing that $\widehat{\chF{\set{U}_\varepsilon}}$ is the $\operatorname{sinc}$-function which is uniformly bounded by the volume of $\set{U}_\varepsilon$ we further get}
    \norm{f_n}[\barron{\beta,c}[\set{U}]]&
    \leq
    \frac{1}{n}
    \int_{\rr{d}}
        \omega(\xi)
        \abs*{
            \widehat{\chF{\set{U}_\varepsilon}}{\left(\frac{\xi-n}{n}\right)}
            \widehat{\rho_\varepsilon}\left(\frac{\xi-n}{n}\right)}
    \dee \xi
    \\&
    \lesssim
    \frac{\abs{\set{U}_{\varepsilon}}}{n}
    \int_{\rr{d}}
        \omega(\xi)
        \abs*{
            \widehat{\rho_\varepsilon}\left(\frac{\xi-n}{n}\right)}
    \dee \xi
    \\&
    =
    \frac{\abs{\set{U}_{\varepsilon}}}{n}
    \int_{\rr{d}}
        \omega(\xi+n-n)
        \abs*{
            \widehat{\rho}_\frac{\varepsilon}{n}\left(\xi-n\right)}
    \dee \xi
    \intertext{and finally with the submultiplicativity of $\omega$}
    \norm{f_n}[\barron{\beta,c}[\set{U}]]&
    \lesssim
    \omega\left(n\right)
    \frac{\abs{\set{U}_{\varepsilon}}}{n}
    \int_{\rr{d}}
        \omega\left(\xi-n\right)
        \abs*{
            \widehat{\rho}_\frac{\varepsilon}{n}\left(\xi-n\right)}
    \dee \xi
    =
    \frac{\abs{\set{U}_{\varepsilon}}e^{cn^\beta}}{n}
    \norm{\widehat{\rho}_\frac{\varepsilon}{n}}[\barron{\beta,c}]
    <\infty.
\end{align*}
Thus, $f_n\in\barron{\beta,c}[\set{U}]$.
For the operator-norm of $\partial^k$, however, we see
\begin{align*}    
    \norm{f_n}[L^2(\set{U})]=1
\qquad\text{whereas}\qquad
    \norm{\partial^k f_n}[L^2(\set{U})]
    =
    \norm*{\frac{n^k}{\sqrt{\pi}}\sin\left(nx+\frac{k\pi}{2}\right)}[L^2(\set{U})]=n^k
\end{align*}
diverges for any $k\geq 1$ when $n\to\infty$, thus $\partial^k$ is unbounded.
\end{proof}

A similar argument about the boundedness of the Barron space can be found in \cref{app:structure_barron}.

\subsection{Approximation of Functions With Sobolev-Growth-Condition}
\label{sec:bounded_growth:bounded_derivatives}

Since the exponential spectral Barron space did not provide us with the growth condition on the semi-norms,
we will continue limiting to the class of functions for which this condition is indeed fulfilled.
That is, we will study a class of functions that has bounded derivatives.
\begin{definition}[Bounded-Derivative-Class]
Let $U\subset \rr{d}$ and be $\{M_\ell\}_{\ell=0}^\infty$ a positive sequence in $\rr {d}$. The class
$\Gamma(\set{U};\{M_\ell\}_{\ell=0}^\infty)$ consists of
all $f\in C^\infty (\rr {d})$ such that
\begin{equation*}
\norm{\partial^\alpha f}[L^2(\set{U})] \lesssim  M_{\abs{\alpha}}\norm{f}[L^2\set(U)].
\end{equation*}
\end{definition}
In the remainder of this section
we will analyze two special cases of this class that appear in the literature and we discuss their properties.

\subsubsection{Weighted Gevrey Class}
\label{sec:gevrey}

Similar to the multipliers class, another class of symbols that appears in the literature is the (weighted) class of Gevrey symbols \cite{Abdeljawad20LiftingsUltramodulationSpaces} which we define as follows.
\begin{definition}[Weighted Gevrey class]
Let $s\ge 0$ and $\omega$ is a weight function defined on $\rr {d}$. The class
$\Gamma^{(\omega)}_s(\rr {d})$ ($\Gamma ^{(\omega )}_{0,s}(\rr {d})$) consists of
all $f\in C^\infty (\rr {d})$ such that
\begin{equation*}
    \abs{\partial^\alpha f(x)} \lesssim  h^{\abs{\alpha}}\alpha !^s\omega (x),  \qquad x\in\rr{d},
    \end{equation*}
    for {\it some} $h>0$ (for {\it every} $h>0$).
\end{definition}
Note that this definition is is broader than
\cite[Definition 1.6]{Abdeljawad20LiftingsUltramodulationSpaces} in the sense that we do not ask for moderateness of the weight.
By doing so, we remove the restriction that
\begin{align*}
    e^{-r\abs{x}}\lesssim \omega(x)\lesssim e^{r\abs{x}},  \text{ for any } x\in \rr{d} \text{ and some } r>0,
\end{align*}
which enables us to use the Gaussian function $g(x) = e^{-\abs{x}^2}$ as a weight.
In the current section, we limit our analysis to the functions $f$ that are self-weighted Gevrey symbols, i.e.,
\begin{align*}
    \Gamma_s=\{f\in C^\infty(\set{U}):f\in\Gamma_s^{(f)}\}
\end{align*}
with $\set{U}\subset\rr{d}$.

As a first step, we now show that this class is indeed non-trivial, in the sense that it not only consists of the zero-function.

\begin{lemma}
\label{exa:gauss}
    For a bounded domain $\set{U}\subset\rr{d}$ with $R_{\set{U},i}=\sup_{x\in\set{U}}\abs{x_i}$ and $\abs{R_{\set{U}}}_2^2 = {\sum_{i=1}^dR_{\set{U},i}^2}$ we get
    \begin{align*}
        \abs{\partial^\alpha f(x)}
        &
        \leq
        e^{\frac{\abs{R_{\set{U}}}_2^2}{2}}
        k^d
        \sqrt{2}^{\abs{\alpha}}
        \sqrt{\abs{\alpha}!}
        \abs{f(x)}
    \end{align*}
    for $f(x)=e^{-\abs{x}^2}$.
    Therefore, $f\in\Gamma_s$.
\end{lemma}
\begin{proof}
    Let $d\geq 1$ and $f(x)=\prod_{i=1}^df_i(x_i)$ with $f_i(x_i)=e^{-x_i^2}$,
    then
    \begin{align*}
        \partial^\alpha f(x)
        =\prod_{i=1}^d f_i^{(\alpha_i)}(x_i)
        =(-1)^{\abs{\alpha}}f(x)\prod_{i=1}^dH_{\alpha_i}(x_i),
    \end{align*}
    where $H_n$ is the $n$-th Hermite polynomial.
    This leads to the bound (see \cite[Vol II, Page 208]{Bateman53HigherTranscendentalFunctions})
    \begin{align*}
        \abs{f_i^{(\alpha_i)}(x_i)}
        &=
        \abs{H_{\alpha_i}(x_i)}\abs{f_i(x_i)}
        \leq
        k\sqrt{\alpha_i!}\sqrt{2}^{\alpha_i}e^{\frac{x_i^2}{2}}\abs{f(x_i)},
    \end{align*}
    with $k$ being Charlier's constant.
    For a bounded domain with $R_{\set{U},i}=\sup_{x\in\set{U}}\abs{x_i}$ we can thus develop the uniform bound
    \begin{align*}
        \abs{f_i^{(\alpha_i)}(x_i)}
        \leq
        k\sqrt{\alpha_i!}\sqrt{2}^{\alpha_i}e^{\frac{R_{\set{U},i}^2}{2}}\abs{f(x_i)}
    \end{align*}
    which is
    \begin{align*}
        \abs{\partial^\alpha f(x)}
        &=
        \prod_{i=1}^d\abs{f_i^{(\alpha_i)}(x_i)}
        \leq
        \prod_{i=1}^d
            e^{\frac{R_{\set{U},i}^2}{2}}
            k\sqrt{\alpha_i!}
            \sqrt{2}^{\alpha_i}
            \abs{f(x_i)}
        \\&
        =
        e^{\frac{\abs{R_{\set{U}}}_2^2}{2}}
        k^d
        \sqrt{2}^{\abs{\alpha}}
        \abs{f(x)}
        \prod_{i=1}^d
            \sqrt{\alpha_i!}
        \\&
        \leq
        e^{\frac{\abs{R_{\set{U}}}_2^2}{2}}
        k^d
        \sqrt{2}^{\abs{\alpha}}
        \sqrt{\abs{\alpha}!}
        \abs{f(x)}
    \end{align*}
    where we use the notation
    $\abs{R_{\set{U}}}_2^2 = {\sum_{i=1}^dR_{\set{U},i}^2}$ and
    $\abs{\alpha} = \sum_{i=1}^d\abs{\alpha_i}$.
\end{proof}

Note that the initial assumption for this class of function was very restrictive, which might be impractical.
Some more observations and insights on the structure of this class of functions can be found in \cref{app:structure_gevrey}.

\subsubsection{Barron-Bandlimited Functions}

In the current section we consider approximation over the unbounded domain $\set{U}=\rr{d}$ and we restrict the target functions to certain types of bandlimited functions.
That is, function whose frequency spectrum is confined to a finite range of frequencies. In other works, a function $f$  is said to be bandlimited with bandwidth $\Omega>0$ if $\operatorname{supp}\hat{f}\subseteq [-\Omega,\Omega]^d$.
For these functions we get the growth condition of \cref{thm:approximation_Frechet_metric:downwards} in the form
\begin{align*}
    \norm{f}[H^k]
    &
    =
    \norm{\eabs{\cdot}^k\hat{f}}[L^2]
    \leq
    \eabs{\Omega}^k\norm{\hat{f}}[L^2]
    =
    \eabs{\Omega}^k\norm{f}[L^2].
\end{align*}

Approximation of $f$ is then equivalent to approximation of $\hat{f}$ over the bounded domain $[-\Omega,\Omega]^d$ due to Parsevals theorem and the linearity of the Fourier transform
\begin{align*}
    \norm{f-f_N}[L^2]
    =
    \norm{\F\{f-f_N\}}[L^2]
    =
    \norm{\hat{f}-\hat{f}_N}[L^2([-\Omega,\Omega]^d)].
\end{align*}
In order to obtain an approximation rate, we also need to impose some smoothness assumption on the function in the frequency domain.
For doing so, we introduce the notion of Barron-Bandlimited functions.
\begin{definition}[Barron-Bandlimited Functions]
    Let $s> 1$, $\Omega>0$.
    We define the class of Barron-Bandlimited functions $\barron*{s}[\Omega]$ as
    \begin{align*}
        \barron*{s}[\Omega]
        :=
        \left\{
            f\in L^1(\rr{d})
        \middle|
            \operatorname{supp}\hat{f}\subseteq [-\Omega,\Omega]^d\,\text{with}\,\hat{f}\in\barron{1}[[-\Omega,\Omega]^d]
        \right\}
    \end{align*}
    and
    \begin{align*}
        \norm{f}[\barron*{s}[\Omega]]:=\norm{\hat{f}}[\barron{s}[[-\Omega,\Omega]^d]].
    \end{align*}
\end{definition}
With that we get the following explicit approximation rate for functions within the Barron-Bandlimited spaces.
\begin{proposition}[$L^2$-Approximation of Barron-Bandlimited Functions]
\label{prop:barron_bandlimited:L2}
    Let $s\in\nn{}_{\geq 2}$, $\Omega>0$,
    and consider a function $\sigma\in W^{s,1}(\rr{})$ and let 
    \begin{align*}
        \Sigma_{N}
        :=
        \left\{
            \sum_{n=1}^N
            \FourierInv{\ch{[-\Omega,\Omega]^d}{\cdot}\hat{\sigma}(\innerProd{w_n}{\cdot}+b_n)}
            \middle|
            w_n\in\rr{d},b_n\in\rr{}
        \right\}
    \end{align*}
    denote the approximation class corresponding to $\hat{\sigma}$.
    Then, for every $f\in\barron*{1}[\Omega]$
    \begin{align*}
        \inf_{f_N\in\Sigma_{N}}
        \norm{f-f_N}[L^2(\rr{d})]
        &
        \lesssim
        N^{-\frac{1}{2}}\norm{f}[\barron*{1}[\Omega]].
    \end{align*}
\end{proposition}

The previous result highlights the expressivity of the dictionary $\Sigma_N$, showing that shallow networks with activation defined through an affine shift in the Fourier domain and low-pass filtering can approximate Barron-Bandlimited functions with the standard Monte Carlo rate of order $-\frac{1}{2}$.
Note that the rate does not depend on the dimension of the problem, that is, our result is free from the curse of dimensionality.

\begin{proof}
    For $\sigma\in W^{s,1}(\rr{})\subset L^1(\rr{})$ the Fourier transform exists in closed form and $\hat{\sigma}\in L^\infty(\eabs{\cdot}^s;\rr{})$.
    Furthermore, the multiplication with the characteristic function enforces that the resulting function is again in $L^1$ and the inverse Fourier transform can be evaluated in closed form.
    Therefore, the approximation class $\Sigma_N$ is well defined. 
    We now define the class
    \begin{align*}
        \widehat{\Sigma}_{N}
        :=
        \left\{
            \sum_{n=1}^N
            \hat{\sigma}(\innerProd{w_n}{\cdot}+b_n)
        \middle|
            w_n\in\rr{d},
            b_n\in\rr{}
        \right\}.
    \end{align*}
    According to \cite[Theorem 2]{Siegel20ApproximationRatesNeural}, the approximation class $\widehat{\Sigma}_N$ is suitable for approximating $g\in\barron{1}[[-\Omega,\Omega]^d]$ in $L^2([-\Omega,\Omega]^d)$ with dimension-independent rate,
    that is,
    \begin{align*}
        \inf_{g_N\in\widehat{\Sigma}_{N}}
        \norm{g-g_N}[L^2([-\Omega,\Omega]^d)]
        \lesssim
        N^{-\frac{1}{2}}\norm{g}[\barron{1}[[-\Omega,\Omega]^d]].
    \end{align*}
    For $f\in\barron*{1}[\Omega]$ we can identify $g$ with $\hat{f}$ which has bounded support and extend the error norm to the full $\rr{d}$
    \begin{align*}
        \norm{\hat{f}-g_N}[L^2([-\Omega,\Omega]^d)]
        =
        \norm{\hat{f}-g_N\chF{[-\Omega,\Omega]^d}}[L^2(\rr{d})].
    \end{align*}
    Consequently, with Parseval's theorem (note that $\hat{f}$ and $g_N\chF{[-\Omega,\Omega]^d}$ are $L^2(\rr{d})$)
    \begin{align*}
        \norm{\hat{f}-g_N}[L^2([-\Omega,\Omega]^d)]
        =
        \norm{f-\FourierInv{g_N\chF{[-\Omega,\Omega]^d}}}[L^2(\rr{d})].
    \end{align*}
    We can view the approximation class $\Sigma_N$ as the inverse Fourier transform of $\widehat{\Sigma}_N$ in the sense that $\FourierInv{g_N\chF{[-\Omega,\Omega]^d}}\in\Sigma_N$ for all $g_N\in\widehat{\Sigma}_N$.
    Conversely, for every $h_N\in\Sigma_N$ there is a (possibly non-unique) parametrization $\{(w_n,b_n)\}_{n=1}^N$ and a corresponding $g_N\in\widehat{\Sigma}_N$ with the same parametrization for which $\Fourier{h_N}=\chF{[-\Omega,\Omega]^d}g_N$.
    Thus,
    \begin{align*}
        \inf_{f_N\in\Sigma_{N}}
        \norm{\hat{f}-\Fourier{f_N}}[L^2(\rr{d})]
        =
        \inf_{g_N\in\widehat{\Sigma}_{N}}
        \norm{\hat{f}-g_N\chF{[-\Omega,\Omega]^d}}[L^2(\rr{d})].
    \end{align*}
    Combining all the above leads to
    \begin{align*}
        \inf_{f_N\in\Sigma_{N}}
        \norm{f-f_N}[L^2(\rr{d})]
        &
        =
        \inf_{f_N\in\Sigma_{N}}
        \norm{\hat{f}-\Fourier{f_N}}[L^2(\rr{d})]
        \\&
        =
        \inf_{g_N\in\widehat{\Sigma}_{N}}
        \norm{\hat{f}-g_N\chF{[-\Omega,\Omega]^d}}[L^2(\rr{d})]
        \\&
        =
        \inf_{g_N\in\widehat{\Sigma}_{N}}
        \norm{\hat{f}-g_N}[L^2([-\Omega,\Omega]^d)]
        \\&
        \lesssim
        N^{-\frac{1}{2}}\norm{\hat{f}}[\barron{1}[[-\Omega,\Omega]^d]]
        =
        N^{-\frac{1}{2}}\norm{f}[\barron*{1}[\Omega]].
    \end{align*}
\end{proof}
The approximation result for Barron-Bandlimited functions in the \Frechet{} metric $d_{\barron*{s}[\Omega]}$
then reads as follows.
\begin{theorem}[\Frechet{}-Approximation of Barron-Bandlimited Functions]
\label{prop:barron_bandlimited:Frechet}
    Let $s> 1$, $\Omega\in\rr{}$, $\Omega>0$,
    and consider a function $\sigma\in W^{s,1}(\rr{})$ and let 
    \begin{align*}
        \Sigma_{N}
        :=
        \left\{
            \sum_{n=1}^N
            \FourierInv{\ch{[-\Omega,\Omega]^d}{\cdot}\hat{\sigma}(\innerProd{w_n}{\cdot}+b_n)}
            \middle|
            w_n\in\rr{d},b_n\in\rr{}
        \right\}
    \end{align*}
    denote the approximation class corresponding to $\hat{\sigma}$.
    Then, for every $f\in\barron*{1}[\Omega]$ and any $\varepsilon>0$
    it is sufficient to choose
    \begin{align*}
        N
        \geq
        N_{\ell_\varepsilon}
        :=
        \left\lceil 
           \left(
                \frac{2^{-\ell_\varepsilon}}{\norm{f}[\barron*{1}[\Omega]]M_{\ell_\varepsilon}}
            \right)^{-2}
        \right\rceil
        \qquad\text{with}\qquad
        \ell_\varepsilon:=\lceil-\log_2(\varepsilon)\rceil+1
    \end{align*}
    with $M_{\ell_\varepsilon}=\eabs{\Omega}^{\ell_\varepsilon}$
    in order to achieve
    \begin{align*}
        \inf_{f_\nNeuron\in\Sigma_\nNeuron} d_{\targetSpace{}}(f-f_\nNeuron)
        =
        \inf_{f_\nNeuron\in\Sigma_\nNeuron}
            \sum_{\ell=0}^\infty\frac{1}{2^\ell}
            \frac{\seminorm{f-f_\nNeuron}[\ell]}
                {1+\seminorm{f-f_\nNeuron}[\ell]}
        < \varepsilon.
    \end{align*}
\end{theorem}
\begin{proof}
    This result follows immediately from \cref{thm:approximation_Frechet_metric:downwards,prop:barron_bandlimited:L2}, where $r(N) = N^{-\frac{1}{2}}$ and $C_f=\norm{f}[\barron*{1}[\Omega]]$.
\end{proof}
The previous theorem provides a lower bound on the needed number of neurons $N$ in order to achieve a predefined approximation error $\epsilon$. The approximation error is quantified in  a \Frechet{} metric which provides a fine-grained measure of approximation quality. It is worth mentioning that the result does not suffer from the curse of dimensionality.

\printbibliography

\appendix

\section{Observations on the Exponential Spectral Barron Space}
\label{app:structure_barron}

\begin{proposition}
    The $L^2(\set{U})$-unit ball is unbounded in $\barron{\beta,c}[\set{U}]$ when $\beta\in(0,1)$, $c>0$ and $\set{U}\subseteq \rr{d}$.
\end{proposition}

Let $f\in \barron{\beta,c}[\set{U}]$, we first extend the pointwise bound
\begin{align}
\label{eq:counterexample:uniformMlBound}
    \abs{f_n^{(k)}(x)}
    \leq
    \norm{f_n}[\barron{\beta,c}]
    \left(\frac{1}{c\beta}\right)^{\frac{k}{\beta}}\left(k!\right)^\frac{1}{\beta},
\end{align}
from \cref{lem:symbol_embedding} on the $k$-th derivative to the $L^2$-bound
\begin{align}
\label{eq:counterexample:L2MlBound}
    \norm{f_n^{(k)}}[L^2(\set{U})]
    \leq
    \abs{\set{U}}^{\frac{1}{2}}
    \norm{f_n}[\barron{\beta,c}]
    \left(\frac{1}{c\beta}\right)^{\frac{k}{\beta}}\left(k!\right)^\frac{1}{\beta}.
\end{align}
For $\set{U}=[0,2\pi]$ and the sequence of functions $\{f_n\}_{n\in\nn{}}$, $f_n(x)=\frac{1}{\sqrt{\pi}}\cos(nx)$ we have for all $k\in\zzp{}$
\begin{align*}
    n^k
    =
    \norm{f_n^{(k)}}[L^2(\set{U})]
    \leq
    \abs{\set{U}}^{\frac{1}{2}}
    \norm{f_n}[\barron{\beta,c}]
    \left(\frac{1}{c\beta}\right)^{\frac{k}{\beta}}\left(k!\right)^\frac{1}{\beta}
\end{align*}
and equivalently
\begin{align*}
    \norm{f_n}[\barron{\beta,c}]
    \geq
    \frac{1}{
    \abs{\set{U}}^{\frac{1}{2}}}
    \left(n^\beta c\beta\right)^\frac{k}{\beta}
    \left(\frac{1}{k!}\right)^\frac{1}{\beta}.
\end{align*}
This expression is increasing over $k$ as long as $n^\beta c\beta\leq k$ and decreasing otherwise.
Thus, its maximum is attained for $k=\lfloor n^\beta c \beta\rfloor$ and we get the lower bound
\begin{align*}
    \norm{f_n}[\barron{\beta,c}]
    &\geq
    \frac{1}{\abs{\set{U}}^{\frac{1}{2}}}
    \left(n^\beta c\beta\right)^\frac{\lfloor n^\beta c \beta\rfloor}{\beta}
    \left(\frac{1}{\lfloor n^\beta c \beta\rfloor!}\right)^\frac{1}{\beta}
    \\&
    \geq
    \frac{1}{\abs{\set{U}}^{\frac{1}{2}}}
    \left\lfloor n^\beta c\beta\right\rfloor^\frac{\lfloor n^\beta c \beta\rfloor}{\beta}
    \left(\frac{1}{\lfloor n^\beta c \beta\rfloor!}\right)^\frac{1}{\beta}
    \\&
    \geq
    \frac{1}{\abs{\set{U}}^{\frac{1}{2}}}
    \left(
    \frac{1}{\sqrt{2\pi}\left\lfloor n^\beta c\beta\right\rfloor}
    e^{\left\lfloor n^\beta c\beta\right\rfloor-1}
    \right)^\frac{1}{\beta},
\end{align*}
where we used the Stirling bound in the last step.
The lower bound diverges for $n\to\infty$, thus, the sequence of functions is unbounded in $\barron{\beta,c}[\set{U}]$ while $\norm{f_n}[L^2(\set{U})]=1$.

\section{Observations on the Weighted Gevrey Class}
\label{app:structure_gevrey}

In \cref{sec:gevrey} we studied the self-weighted Gevrey class which fulfils the condition
\begin{align*}
    \abs{\partial^\alpha f(x)}\lesssim h^{\abs{\alpha}}\abs{\alpha} !^sf(x).
\end{align*}
We now provide some insight into the structure of this class.

For $d=1$ the pointwise inequality means that for any $x_0$ with $f(x_0)=0$ we get that all derivatives are zero as well at this position.
An immediate consequence of that is, that the cosine-approximation class is not in this space as not all of its derivatives are zero at all $x_0$.

For the analytic cases with $\beta\geq 1$ this would immediately imply that the analytic extension in this point is the zero function.
Due to the uniqueness of the extension, we can conclude that any function in this space must be either strictly positive/negative or the zero-function.

For the non-analytic cases with $\beta<1$ the class of functions is allowed to contain mollifiers that are zero outside some bounded domain.
This means, if $f(x_0)=0$, then $x_0$ lies between two bump functions and therefore, the class of functions can be a spatial concatenation of bump functions.

\paragraph{Disjoint Bump Functions}
In order to apply the $M_\ell$ inequality on the error $f-f_N$, we need to either show that it holds for the error, or we need to show that it holds for $f$ and $f_N$ and that the space of functions is a vector space.
One way to get the latter together with the Gevrey-class is to build a set of functions that consists of a fixed set of non-overlapping bump-functions, where each individual bump is scaled by a single scalar value.
The consequence of this construction is, that each of these functions can be identified by the corresponding list of scalar values.
Approximation in this space could then for example mean that we have to select $N$ of the scalar values, which we then optimize.

There cannot be uncountably many bump functions. The argument against that is a very simple, standard-argument for the fact that there cannot be uncountably many open intervals in a partition of $\rr{}$.
The support of each bump function must be an open and non-empty interval.
Therefore, it contains at least one rational number.
By mapping from the interval to a selected rational number from within the interval, we have an injective mapping from the intervals to the rational numbers.
Therefore, the inverse exists on a subset of the rational numbers.
The domain of this inverse is countable, therefore, the number of intervals must be countable.

By means of an argument with a homogeneous linear first-order differential equation, we can argue that there is no bump function with bounded support in $\Gamma_s$.

Let $f\in\Gamma_s$, then 
\begin{align*}
    \abs{D^\ell f(x)}\leq M_{\ell} \abs{f(x)}
\end{align*}
this means
we get that $f$ satisfies the following homogeneous differential equation
\begin{align*}
    D^\ell f(x)=g_{f,\ell}(x)f(x), \text{ where }f(x)\neq 0.
\end{align*}
Here we exclude all the points $x$ such that $f(x)=0$, since
when $f(x)=0$ it follows that $D^\ell f(x)=0$ for all $\ell\in\nn{}$.
A valid choice for $g_{f,\ell}$ is 
\begin{align*}
    g_{f,\ell}(x)
    =
    \begin{cases}
        \frac{D^\ell f(x)}{f(x)},&f(x)\neq 0\\
        0,&\text{otherwise},
    \end{cases}
\end{align*}
We observe that $g_{f,\ell}$ is bounded
(more precisely, $\abs{g_{f,\ell}(x)}\leq M_{\ell}$)
and it has antiderivative $G_{f,\ell}$ almost everywhere (as $f\in C^\infty)$.
With that, we can write
\begin{align*}
    f(x)=
    \begin{cases}
        e^{G_{f,1}(x)+K}, & f(x)\neq 0,\\
        0, &\text{otherwise}
    \end{cases}
\end{align*}
for some $K\in\rr{}$.

Assume that $f(x_0)=0$ for some finite $x_0$,
then due to the continuity of $G_{f,\ell}$
we have $\lim_{x\to x_0}G_{f,1}(x)=-\infty$
which is a contradiction to the boundedness of $g_{f,1}$ by the fundamental theorem of calculus.

\begin{proposition}
    Weighted Gevrey class of functions is either trivial or intractable to approximate.
\end{proposition}
\begin{proof}
    The proof could be found in the Appendix, where we provide examples and discussions regarding the triviality of using weighted Gevrey regular functions.
    Further we investigate the fact that using these classes could also be impossible to approximate in some situations.
\end{proof}

\end{document}